\documentclass [11pt]{article}
\usepackage{amsfonts}
\usepackage{amsmath,amssymb,amsthm}
\usepackage{graphicx}
\usepackage{bbm}
\usepackage[all]{xy}
\usepackage{rotating}
\usepackage{multirow}
\usepackage{enumerate}
\usepackage[dvipsnames]{xcolor}
\usepackage{tikz}
\usetikzlibrary{arrows,intersections,fadings,calc}
\usepackage{todonotes}
\usepackage[a4paper,textwidth=16cm,textheight=24cm]{geometry}
\usepackage[authoryear,round]{natbib}
\usepackage{authblk}

\newtheorem{theorem}{Theorem}
\newtheorem{remark}{Remark}
\newtheorem{lemma}{Lemma}
\newtheorem{proposition}{Proposition}
\newtheorem{corollary}{Corollary}
\newtheorem{definition}{Definition}
\newtheorem{example}{Example}

\newcommand{\gos}{GOS}

\title{Integral stochastic orders of $m$-generalized order statistics from transform-ordered nonparametric families}

\author[1]{Idir Arab\thanks{idir.bhh@gmail.com, ORCID: 0000-0002-9408-684X}}

\author[2]{Tommaso Lando\thanks{tommaso.lando@unibg.it, ORCID: 0000-0003-4288-0264}}

\author[1]{Paulo Eduardo Oliveira\thanks{Email: paulo@mat.uc.pt, ORCID: 0000-0001-7217-5705}\thanks{P.E.O. acknowledges financial support by the Centre for Mathematics of the University of Coimbra (CMUC, https://doi.org/10.54499/UID/00324/2025)
under the Portuguese Foundation for Science and Technology (FCT),
Grants UID/00324/2025 and UID/PRR/00324/2025.}
}

\author[3]{Tomasz Rychlik\thanks{trychlik@impan.pl, ORCID: 0000-0002-4298-8139}}

\affil[1]{CMUC, Department of Mathematics, University of Coimbra, Portugal}

\affil[2]{Department of Economics, University of Bergamo, Italy}

\affil[3]{Institute of Mathematics, Polish Academy of Sciences, Poland }

\date{}
\begin{document}

\maketitle

\begin{abstract}
We provide sufficient conditions for comparing $m$-generalized order statistics with respect to the increasing concave, increasing convex, and star-shaped stochastic orders. These conditions allow us to rank classical order statistics, 
selected censored type-II order statistics, and records. They depend on both the parameters of the generalized order statistics and the underlying distribution. Rather than assuming a specific parametric form, we adopt a nonparametric approach and 
assume some stochastic transform-ordered property, that is, some suitable shape condition. 
This framework encompasses many relevant classes of distributions that are related, via transform order, to the generalized and the negative generalized Pareto distribution.

\medskip

\noindent Mathematics Subject Classification 2020: 60E15, 62G30, 62G32.

\medskip

\noindent
Keywords and phrases: generalized order statistics, $m$-generalized order statistics, $k$-th record values, integral stochastic order, transform stochastic order.

\end{abstract}

\section{Introduction}

One of the main problems in the theory of stochastic orders is to develop tools for comparing random variables arising from sampling. In general, this category includes any random variable defined as a function of the sample, such as classical statistical functionals including the sample mean and variance, linear combinations, order statistics, and sample spacings. Conditions for comparing such random objects may depend both on the type of function under consideration and the underlying distribution of the sample. However, while the functional form is known, the underlying distribution is typically unknown. 
This motivates the search for conditions on the underlying distribution under which stochastic comparisons can be derived.

In this paper, we focus on the stochastic comparison of generalized order statistics introduced by \cite{Kamps1995}. This model provides a unified framework for a number of notions arising in applied probability, reliability, survival analysis, and statistics, including ordinary order statistics, censored type-II order statistics, and record values.
In this regard, stochastic comparisons of order statistics represent a classical problem, for which many results have been obtained, depending on the stochastic order under consideration and on the assumptions imposed on the underlying distribution; see, among others, \cite{ALO-JAP}, \cite{arnold1991exp}, \cite{kochar2006, kochar2012}, \cite{kochar2009}, \cite{kundu2016}, \cite{lando2021st}, and \cite{wilfling1996c}. These results have natural applications in reliability theory, where an order statistic represents the lifetime of a $k$-out-of-$n$ system with i.i.d.\ components.

While most of the existing literature derives ordering conditions under specific parametric assumptions on the underlying distribution, \cite{lando2021st} and, more recently, \cite{ALO-JAP}, obtain results based on integral stochastic orders (see \cite{muller1997} for the general definition) over broad nonparametric families of distributions defined through suitable transform stochastic orders (see \cite{shaked2007}). As noted by \cite{ALO-JAP}, a suitable choice of the integral order and the corresponding transform order yields explicit comparison conditions for order statistics within the relevant transform-ordered family. In the present paper, we show that a similar idea extends to the more general setting of generalized order statistics. This extension is mathematically nontrivial, since generalized order statistics depend on vectors of parameters whose interaction substantially complicates the analysis.

Integral stochastic orders include the usual stochastic order, the increasing concave order, the increasing convex order, and the star-shaped order (see \cite{shaked2007}). These orders are of particular interest because they allow one to compare random variables simultaneously in terms of magnitude and dispersion. Extending this perspective to generalized order statistics gives rise to several interesting problems, some of which have not yet been addressed. For example, beyond ordinary order statistics, censored type-II order statistics describe the time of the $k$-th failure in a life test that starts with $N$ units and stops after $n$ failures, under a given censoring scheme. Comparing such failure times across different tests, with different values of $N$, $n$, and different censoring schemes, in terms of both size and dispersion, {is naturally} useful in applications. Indeed, such comparisons can help determine under which testing design failures tend to occur later and exhibit greater or smaller variability. Similarly, in the context of record values, one may compare the $k$-th and $j$-th upper records arising from sequences of $n$ and $m$ record observations from the same underlying distribution.

We provide general conditions for establishing integral stochastic orderings of $m$-generalized order statistics, a subclass of generalized order statistics that includes the models mentioned above. Our results strictly contain, as special cases, those of \cite{ALO-JAP} and \cite{lando2021st}. As discussed above, these results apply to transform-ordered families of distributions. Such families include broad and important classes that already play a significant role in reliability, survival analysis, probability, and statistics. In particular, we obtain conditions for families of distributions that are described through a stochastic transform ordering with respect to the generalized Pareto distribution. This includes distributions with monotone density, monotone density on average, increasing or decreasing failure rate (\cite{marshall2007}), decreasing failure rate on average (\cite{marshall2007}), increasing or decreasing odds rate (\cite{oddspaper} or \cite{superpareto2025}), among many others.

\section{Preliminaries and main theorem}

Generalized order statistics (see \cite{Kamps1995}), 
referred to on the sequel as \gos, is a semiparametric probability model based on a vector $(\gamma_1, \ldots, \gamma_n)$ or a sequence $(\gamma_1,\gamma_2, \ldots)$ of arbitrary positive parameters and a one-dimensional baseline distribution function $F$.
This is defined in a two-stage procedure. A vector of random variables $(U_{1,\tilde{\gamma}_1}, \ldots,U_{n,\tilde{\gamma}_n})$  (sequence $(U_{1,\tilde{\gamma}_1},U_{2,\tilde{\gamma}_2},\ldots)$, respectively) is called a vector (sequence, respectively) of uniform generalized order statistics if
it has a joint distribution identical with
\begin{equation}
U_{r,\tilde{\gamma}_r} \stackrel{d}{=}1- \prod_{i=1}^r U_i^{\frac{1}{\gamma_i}}, \qquad r=1,\ldots,n \;
(r=1,2,\ldots,\;\mathrm{respectively}),
\label{1}
\end{equation}
where $U_1, \ldots,U_n$ ($U_1,U_2,\ldots$, respectively) are i.i.d.\
standard uniform random variables. We see that
the distribution of the first $r$ uniform generalized order statistics depends only on the first parameters $\gamma_1,\ldots,
\gamma_r$.
Accordingly, the subscript $(r, \tilde{\gamma}_r)$ in our notation means the number of \gos\ 
and respective vector of significant parameters $\tilde{\gamma}_r=
(\gamma_1, \ldots, \gamma_r)$.
The vector (sequence) of \gos\ 
$(X_{1,\tilde{\gamma}_1}, \ldots,
X_{n,\tilde{\gamma}_n})$  ($(X_{\tilde{\gamma}_1}, X_{2,\tilde{\gamma}_2},\ldots)$, respectively), with
parameters $(\gamma_1, \ldots, \gamma_n)$
($(\gamma_1,\gamma_2, \ldots)$, respectively), and a baseline
distribution function $F$ has the distribution as
\begin{equation}
X_{r,\tilde{\gamma}_r} \stackrel{d}{=}F^{-1}
\left( 1- \prod_{i=1}^r U_i^{\frac{1}{\gamma_i}}\right),
\qquad r=1,\ldots,n \;
(r=1,2,\ldots,\;\mathrm{respectively}).
\label{2}
\end{equation}
According to (\ref{1}) and (\ref{2}), the distribution of the first $r$ \gos\ depends only on the first $r$ parameters and the baseline distribution. Moreover, the marginal distribution of $r$-th \gos\ does not depend on the ordering of these parameters.

The most popular submodel of \gos\ is the vector of standard order statistics based on $n$ i.i.d.\ random variables with a marginal distribution function $F$. These are just \gos\ 
with parameters $\gamma_r= n+1-r$, $r=1,\ldots,n$. Its generalization constitutes the model of progressively censored type-II order statistics $X_{1:n:N}^{\mathbf{R}}, \ldots,X_{n:n:N}^{\mathbf{R}}$ with $N$ objects under experiment, $n \leq N$ actual observations,
and censoring plan $\mathbf{R}=(R_1, \ldots,R_n)$, where $R_i$ is the number of randomly chosen objects removed from the experiment just after $X_{i:n:N}^{\mathbf{R}}$  so that $R_i \geq 0$ and $\sum_{i=1}^n R_i +n =N$. This undergoes the \gos\ model with a finite sequence of parameters
$\gamma_r= \sum_{i=r}^n (R_i+1)$, $r=1,\ldots,n$.
Another classic example of \gos\ 
is the sequence of $k$-th upper records $R_1^{(k)}, R_2^{(k)}, \ldots$ with parameters $\gamma_i=k \in \mathbb{N}$ and a continuous baseline distribution function $F$. 
In particular, for $k=1$, we obtain the classic upper record values.

Marginal distributions of \gos, 
and in particular uniform ones are complicated in general, especially if some parameters (not all) are multiple.
Recursive formulas were presented in \cite{CramerKamps2003}. Explicit expressions appeared in \cite{Rychlik2026}. They simplify much in the case of so called
$m$-generalized order statistics ($m$-\gos) for which the parameters form arithmetic sequences.
We can construct an infinite sequence of $m$-\gos\ if the common  difference is nonnegative. Otherwise, the assumption of positivity of all the gammas allows the construction of only a finite number of $m$-\gos.
The assumption of the $m$-\gos\ model is  satisfied by standard order statistics and $k$-th record values.
The progressively censored type-II order statistics satisfy this submodel if $R_1= \cdots = R_{n-1}$, i.e., if the same number of objects is removed after each failure, except possibly for the last one.
It follows that we can construct a finite number of classic and progressively censored order statistics, and an infinite number of record values.


Following \cite{ALO-JAP}, we consider two families of stochastic orders. Let $\mathcal{H}$ denote a family of nondecreasing functions.
\begin{definition}[\cite{muller1997}]
We say that a random variable $X$ succeeds another random variable $Y$ in $\mathcal{H}$-integral order (denoted as $X \succeq_{\mathcal{H}}^I Y$) if
$\mathbb{E}h(X) \geq \mathbb{E}h(Y)$ for all $h \in \mathcal{H}$ such that the respective expectations exist.
We refer similarly to the order between the respective distribution functions.
\end{definition}

\begin{definition}[\cite{marshall2007}]
We say that a distribution function $F$ succeeds another one $G$ in
$\mathcal{H}$-transform order (shortly $F \succeq_{\mathcal{H}} ^T G$) if $F^{-1} \circ G \in \mathcal{H}$. The same relation will be applied to random variables $X$ and $Y$ having distribution functions $F$ and $G$, respectively.
\end{definition}

The most popular examples of $\mathcal{H}$-integral and
$\mathcal{H}$-transform orders are presented below.
\begin{definition}
Let $F \succeq_{\mathcal{H}}^I G$ for some $\mathcal{H}$.
We say that $F$ succeeds $G$ in
\begin{enumerate}[(i)]
\item the usual stochastic order ($F \succeq_{st} G$) if
$\mathcal{H}$ is the family of nondecreasing functions,

\item the increasing convex order ($F \succeq_{icx} G$) if
$\mathcal{H}$ is the family of nondecreasing convex functions,

\item the increasing concave order ($F \succeq_{icv} G$) if
$\mathcal{H}$ is the family of nondecreasing concave functions,

\item the star-shaped order ($F \succeq_{ss} G$) if
$\mathcal{H}$ is the family of nondecreasing star-shaped functions,

\item the anti-star-shaped order ($F \succeq_{ias} G$) if
$\mathcal{H}$ is the family of nondecreasing
anti-star-shaped functions.
\end{enumerate}
\end{definition}

\begin{definition}
Let $F \succeq_{\mathcal{H}}^T G$ for some $\mathcal{H}$.
We say that $F$ succeeds $G$ in
\begin{enumerate}[(i)]
\item the usual stochastic order ($F \succeq_{st} G$) if
$\mathcal{H}$ is the family of nondecreasing functions not smaller than the identity function,

\item the convex transform order ($F \succeq_{c} G$) if
$\mathcal{H}$ is the family of nondecreasing convex functions,

\item the star order ($F \succeq_{*} G$) if
$\mathcal{H}$ is the family of nondecreasing star-shaped functions,

\item the superadditive order ($F \succeq_{+} G$) if
$\mathcal{H}$ is the family of nondecreasing
superadditive functions.
\end{enumerate}
\end{definition}

The most widely used families of life distributions defined by  the transform order relations with a given distribution function $G$ are the ones related to the standard exponential and standard uniform distributions.
When $G$ is the exponential, the relations $F \preceq_c (\preceq_*,\preceq_+) G$ mean that $F$ has an increasing failure rate  $\lambda_F(x) = \frac{f(x)}{1-F(x)}$ (IFR), increasing failure rate on the average $\Lambda_F(x) = \frac 1x \int_0^x \lambda_F(t)dt$ (IFRA) and new-better-than-used (NBU) property which means that
\begin{equation}
\label{1a}
[1-F(x)][1-F(y)] \geq 1-F(x+y)
\end{equation}
for all $x,y \geq 0$ such that $x+y \leq F^{-1}(1)$.
The reversed relations $F \succeq_c (\succeq_*,  \succeq_+) G$ define the DFR, DFRA and NWU families.  For the standard uniform $G$, the relation $F \preceq_c (\preceq_*, \preceq_+) G$ means that $F$ has increasing density (ID), increasing density on the average (IDA), and subadditive distribution function,
respectively. The reversed inequalities define the DD, DDA, and superadditive distributions. A useful generalization of the exponential and uniform distribution functions are the generalized Pareto distribution functions
\begin{equation}
\label{1b}
W_\alpha(x) = \left\{
\begin{array}{lll}
1-(1-\alpha x)^{\frac{1}{\alpha}}, & x >0, & \alpha <0, \\[1.5ex]
1-\exp (-x), & x>0, & \alpha=0, \\[1.5ex]
1-(1-\alpha x) ^{\frac{1}{\alpha}}, & 0 < x < \frac{1}{\alpha},
& \alpha >0.
\end{array}
\right.
\end{equation}
The function $\lambda_{F,\alpha}=(W_\alpha^{-1}(F))'=f(x) [1-F(x)]^{\alpha-1}$ is called the $\alpha$-generalized failure rate. Monotonicity of $\lambda_{F,\alpha}$ corresponds to $F$ being convex-ordered with respect to $W_\alpha$. Similarly, we consider distributions that dominate $W_\alpha$ in the star order, which corresponds to $\frac1{x}{\lambda_{F,\alpha}(x)}$ being increasing, where the latter function is referred to as the $\alpha$-generalized failure rate average. So the next definitions follow accordingly.
\begin{definition}
We say $F$ has $\alpha$-increasing (decreasing) failure rate ($\alpha$-IGFR and $\alpha$-DGFR, respectively) if $F\preceq_c(\succeq_c) W_\alpha$. Moreover, $F$ has $\alpha$-decreasing failure rate average ($\alpha$-DGFRA) if $F\succeq_* W_\alpha$.
\end{definition}
For $ \alpha =0$ and $1$ we obtain IFR/DFR and ID/DD distributions, respectively. If $\alpha=-1$, we obtain monotone odds rate distributions, denoted as IOR/DOR (\cite{oddspaper}). A similar possible choice of $G$ is the negative generalized Pareto distribution, namely, $\widetilde{W}_\alpha(x)=1-W_\alpha(-x)$, where $-x$ belongs to the support of $W_\alpha$. This models makes it possible to establish ordering conditions for the increasing/decreasing reversed failure rate families of distributions and further generalizations.

Our main theorem generalizes Theorem 3.1 of \cite{ALO-JAP}. The arguments of its proof are similar.
\begin{theorem}\label{t0}
Let $\mathcal{H}$ be a family of nondecreasing functions that is closed under compositions. Let $X_{r,\tilde{\gamma}_r}$ and
$Y_{q,\tilde{\beta}_q}$ denote the $r$-th and $q$-th \gos\
based on parameters $\tilde{\gamma}_r=(\gamma_1,\ldots,\gamma_r$) and $\tilde{\beta}_r=(\beta_1,\ldots,\beta_q)$, respectively, and baseline distribution functions $F$ and $G$, respectively. If $F \succeq_{\mathcal{H}}^T G$ and $Y_{r,\tilde{\gamma}_r}\succeq_{\mathcal{H}}^I Y_{q,\tilde{\beta}_q}$, then $X_{r,\tilde{\gamma}_r}\succeq_{\mathcal{H}}^I X_{q,\tilde{\beta}_q}$.
\end{theorem}
\begin{proof}
By definition, $Y_{q,\tilde{\beta}_q} \stackrel{d}{=} G^{-1}(
U_{q,\tilde{\beta}_q})$,
$Y_{r,\tilde{\gamma}_r} \stackrel{d}{=} G^{-1}(
U_{r,\tilde{\gamma}_r})$, and $X_{r,\tilde{\gamma}_r} \stackrel{d}{=}
F^{-1}(U_{r,\tilde{\gamma}_r})= F^{-1}\circ G \circ G^{-1}(
U_{r,\tilde{\gamma}_r})$.
Since $\mathbb{E}h(G^{-1}(U_{q,\tilde{\beta}_q})) \leq
\mathbb{E}h(G^{-1}(U_{r,\tilde{\gamma}_r}))$
for all $h \in \mathcal{H}$
for which the expectations exist, $F^{-1}\circ G \in \mathcal{H}$,
and so $h \circ F^{-1}\circ G \in \mathcal{H}$, we have
\begin{equation}
\label{3}
\mathbb{E}h(X_{q,\tilde{\beta}_q}) =
\mathbb{E}h\circ F^{-1}\circ G( G^{-1}(U_{q,\tilde{\beta}_q}))
\leq \mathbb{E}h\circ F^{-1}\circ G( G^{-1}(U_{r,\tilde{\gamma}_r}))
=\mathbb{E}h(X_{r,\tilde{\gamma}_r}).
\end{equation}
This implies $X_{r,\tilde{\gamma}_r}
\succeq_{\mathcal{H}}^I X_{q,\tilde{\beta}_q}$.
\end{proof}

\section{$m$-Generalized order statistics}
\label{sec:m-gos}

In our particular considerations, we focus on orderings among $m$-\gos. With the aim of simplifying calculations, we adopt a specific notation.
Since we analyze fixed single $m$-\gos\ 
for which ordering of the parameters is irrelevant, we assume that the
arithmetic sequence of ordered parameters $(\gamma_{1:r}, \ldots, \gamma_{r:r})$ is nondecreasing.
Then the distribution of $r$-th $m$-\gos\ depends on the number $r\in\mathbb{N}$, minimal parameter $\gamma_{1:r}>0$, common difference
$\mu=\gamma_{i+1:r}-\gamma_{i:r} \geq 0$ (for $i=1,\ldots,r-1$) and baseline distribution function $F$.
Given a set of parameters $r\geq 1$ and $\gamma,\mu>0$, let us denote
\begin{equation*}
\label{eq:Mprod}
M(r,\gamma,\mu)=\prod_{i=1}^r\bigl(\gamma+(i-1)\mu\bigr).
\end{equation*}
According to Lemma~3.1.2 in \cite{Kamps1995}, 
in the standard uniform case it has the density function
\begin{equation}
\label{2a}
f_{r,\gamma_{1:r},\mu}(x) = 
\frac{M(r,\gamma_{1:r},\mu)}{(r-1)!}
(1-x)^{\gamma_{1:r}-1} g^{r-1}_\mu(x), 
\qquad 0<x<1,
\end{equation}
where
\begin{equation}
\label{2b}
g_\mu(x) = \left\{
\begin{array}{ll}
\frac{1}{\mu} \bigl(1-(1-x)^\mu\bigr), & \mu\neq 0, \\[1.5ex]
-\ln(1-x), & \mu = 0.
\end{array}
\right.
\end{equation}
If the baseline distribution function $F$ has a density function $f$, then the density of the corresponding $m$-\gos\ has the form
$f_{r,\gamma_{1:r},\mu}(F(x))f(x)$.
Our notation and convention differ from the ones applied in literature.  It is commonly assumed that the parameter sequence is nonincreasing,
which is consistent with the constructions of models of order
statistics and records. 
Moreover, the common difference $\mu$ is denoted by $m+1$, and the $n$-th parameter $\gamma_n$ is written as $k$. We introduced these modifications for simplifying our calculations.

Below, we characterize the sign variation of the difference between density functions of uniform $m$-\gos.
%
These characterizations provide the background results that will be used in the sequel to prove integral order relations for $m$-\gos\ from restricted families of baseline distributions. 
First, for completeness, we recall a generalized Descartes rule of signs.
\begin{lemma}[\cite{Komornik2006}]
\label{gen-Descartes}
Let $p(x) = a_0x^{b_0} + a_1 x^{b_1} + \cdots + a_n x^{b_n}$ be a function with nonzero real coefﬁcients $a_0,\ldots,a_n$ and real exponents $b_0<\cdots< b_n$.
Then $p$ cannot have more positive roots (even counted with multiplicity) than the number of sign changes in the sequence $a_0,\ldots,a_n$.
\end{lemma}

A first characterization of sign variation when comparing two $r$-th $m$-\gos\ with different parameters.
{\samepage
\begin{lemma}\label{t1}
Let $f_{r,\gamma_{1:r},\mu}$ and $f_{r,\beta_{1:r},\nu}$ denote the density functions of the $r$-th uniform $m$-\gos\, respectively, with the minimal parameters $\gamma_{1:r} > \beta_{1:r}>0$, respectively, and common differences $\mu, \nu \geq 0$, respectively. Consider the following conditions:
    \begin{enumerate}
    \item\label{cc1}
    $\mu, \nu >0$ and $\frac{M(r,\gamma_{1:r},\mu)}{M(r,\beta_{1:r},\nu)}\geq 1$, 
    
    \item\label{cc2}
    $\mu, \nu >0$ and $\frac{M(r,\gamma_{1:r},\mu)}{M(r,\beta_{1:r},\nu)} < 1$, 
    
    \item\label{cc3}
    $\mu=0 <\nu$ and either 
        \begin{enumerate}
        \item
        $\beta_{r:r}\leq\gamma_{1:r}$,
        \item
        $\beta_{r:r}= \beta_{1:r} + (r-1)\nu > \gamma_{1:r}$ together with $\frac{M(r,\gamma_{1:r},0)}{M(r,\beta_{1:r},\nu)}\geq 1$,
        \end{enumerate}
    
    \item\label{cc4}
    $\mu=0 <\nu$ and $\beta_{r:r}>\gamma_{1:r}$ and  $\frac{M(r,\gamma_{1:r},0)}{M(r,\beta_{1:r},\nu)} < 1$,
                    
    \item\label{cc5}
    $\nu=0\leq\mu$.
    \end{enumerate}
    \begin{enumerate}
    \item[1.]
    If either of the conditions (i), (iii) or (v) 
    holds then the difference $f_{r,\gamma_{1:r},\mu}-f_{r,\beta_{1:r},\nu}$ has sign variation $+-$ on $(0,1)$.
    \item[2.]
    If either of the conditions (ii) or (iv) 
    holds then the difference $f_{r,\gamma_{1:r},\mu}-f_{r,\beta_{1:r},\nu}$ has sign variation $-+-$ on $(0,1)$.
    \end{enumerate}
\end{lemma}
}
\begin{proof}
Note that (i) 
and (ii) 
are complementary assumptions on $\mu$ and $\nu$, and the same happens with (iii) 
and (iv), 
so these pairs of cases will be handled simultaneously.
\begin{enumerate}
\item[(a)]
In the cases (i) and (ii) 
both $\mu$  and $\nu$ are positive. We may write
\begin{eqnarray}
\lefteqn{f_{r,\gamma_{1:r},\mu}(x) - f_{r,\beta_{1:r},\nu}(x)} \nonumber \\
 & = &
\frac{M(r,\beta_{1:r},\nu)}{(r-1)!}(1-x)^{\beta_{1:r}-1} \nonumber \\
 & & \times
 \left[\frac{M(r,\gamma_{1:r},\mu)}{M(r,\beta_{1:r},\nu)}
  (1-x)^{\gamma_{1:r}-\beta_{1:r}}
\! \left( \frac{1-(1\!-\!x)^\mu}{\mu}\right)^{r-1}
\! - \left( \frac{1-(1-x)^\nu}{\nu}\right)^{r-1}\right].\qquad 
\label{4}
\end{eqnarray}
The expression in front of the square brackets is positive. The one in the brackets is the difference of two positive terms, and it determines the sign of the whole difference.
We consider the difference of $(r-1)$-th roots of these terms, replacing the original variable $x$ by $y=1-x$, and get
\begin{eqnarray}
h_1(y) & = & 
  \left(\frac{M(r,\gamma_{1:r},\mu)}{M(r,\beta_{1:r},\nu)}\right)^{\frac1{r-1}}
y^{\frac{\gamma_{1:r}-\beta_{1:r}}{r-1}}
\frac{1-y^\mu}{\mu} - \frac{1-y^\nu}{\nu} \nonumber \\
& = & -\frac{1}{\nu}y^0 + \frac{1}{\mu}
 \left(\frac{M(r,\gamma_{1:r},\mu)}{M(r,\beta_{1:r},\nu)}\right)^{\frac1{r-1}}
y^{\frac{\gamma_{1:r}-\beta_{1:r}}{r-1}} \nonumber \\
& & - \frac{1}{\mu} \left(\frac{M(r,\gamma_{1:r},\mu)}{M(r,\beta_{1:r},\nu)}\right)^{\frac1{r-1}}
y^{\frac{\gamma_{1:r}-\beta_{1:r}}{r-1}+\mu} + \frac{1}{\nu}y^\nu.
\label{5}
\end{eqnarray}
Note that $h_1(0) = -\frac{1}{\nu}<0$ and $h_1(1)=0$.
Moreover, easy calculations show that
\begin{equation*}
\label{6}
h'_1(1) = - \left(\frac{M(r,\gamma_{1:r},\mu)}{M(r,\beta_{1:r},\nu)}\right)^{\frac1{r-1}}
 +1,
\end{equation*}
which implies that $h'_1(1) >0 $ and $h_1(1-)<0$ if $\frac{M(r,\gamma_{1:r},\mu)}{M(r,\beta_{1:r},\nu)}<1$,
whereas $h'_1(1) <0 $ and $h_1(1-)>0$ if 
$\frac{M(r,\gamma_{1:r},\mu)}{M(r,\beta_{1:r},\nu)}>1$
Therefore, we need to separate the two cases from now on. 

Before proceeding, note that the exponents in the first three terms of (\ref{5}) are ordered: $0< \frac{\gamma_{1:r}-\beta_{1:r}}{r-1} < \frac{\gamma_{1:r}-\beta_{1:r}}{r-1}+\mu$. Concerning the last exponent $\nu>0$ there are three possible ways to order them:  $0 < \nu < \frac{\gamma_{1:r}-\beta_{1:r}}{r-1}$ or $\frac{\gamma_{1:r}-\beta_{1:r}}{r-1} < \nu <\frac{\gamma_{1:r}-\beta_{1:r}}{r-1} +\mu $ or $\nu > \frac{\gamma_{1:r}-\beta_{1:r}}{r-1} +\mu$.
    \begin{itemize}
   \item
   The case (ii): 
   We are assuming that $\mu, \nu >0$ and $\frac{M(r,\gamma_{1:r},\mu)}{M(r,\beta_{1:r},\nu)}<1$.
    The assumption $\nu \leq \frac{\gamma_{1:r}-\beta_{1:r}}{r-1}+\mu$, i.e. $\beta_{r:r} \leq \gamma_{r:r}$ together with $\gamma_{1:r} > \beta_{1:r}$ contradicts $\frac{M(r,\gamma_{1:r},\mu)}{M(r,\beta_{1:r},\nu)}<1$
    So the only possible order of the exponents in (\ref{5}) is $\nu > \frac{\gamma_{1:r}-\beta_{1:r}}{r-1} +\mu$. 
    In this case the sign sequence of the coefficients is $-+-+$. Since $h_1(1-)<0$, the function (\ref{5}) has sign variation either $-+-$ or $-$ on $(0,1)$ 
    The latter case is impossible because the integral of the difference of density functions (\ref{4}) over $(0,1)$ is equal to $0$, so the sign variation is $-+-$. Finally, note that this sign variation is not affected when we raise both terms in the first line of (\ref{5}) to the power $r-1$ and replace $y$ by $1-x$. Therefore, for this case, the sign variation of (\ref{4}) is $-+-$. 
    \item
    The case (i): 
    We first assume that  $\mu,\nu >0$ and $\frac{M(r,\gamma_{1:r},\mu)}{M(r,\beta_{1:r},\nu)}>1$
    which means $h(1-)>0=h_1(1)$.
    If $\nu <  \frac{\gamma_{1:r}-\beta_{1:r}}{r-1} +\mu$, then (\ref{5}) is 
    $-+-$ in $\mathbb{R}_+$ and $-+$ in $(0,1)$,  and (\ref{4}) is $+-$ there.
    If $\nu >  \frac{\gamma_{1:r}-\beta_{1:r}}{r-1} +\mu$, then (\ref{5}) is either 
    $-+$ or $-+-+$ in $\mathbb{R}_+$ by the rule of signs, and it is so possibly in $(0,1)$. 
    However, $h_1(1+)<0$ contradicts the latter 
    case. 
    Accordingly, (\ref{4}) is $+-$ on $(0,1)$.
    Suppose finally that $\nu =  \frac{\gamma_{1:r}-\beta_{1:r}}{r-1} +\mu$. The sign sequence of the coefficients in (\ref{5}) is either $-+-$ or $-+$.  
    The case $-+$ is excluded because $h_1(1)= 0 < h_1(1+)$ forces $h_1(y)<0$ for all $0<y<1$, which is impossible. 
    Similarly, negativity of $h_1$ on the whole 
    positive half-axis cannot hold. The only possibility is that 
    the function is $-+-$ on $\mathbb{R}_+$, and $-+$ on $(0,1)$. Function (\ref{4}) has the reversed order of signs on the unit interval.   
    
    \smallskip
    
    Suppose now that $\frac{M(r,\gamma_{1:r},\mu)}{M(r,\beta_{1:r},\nu)}=1$. 
    Since $\gamma_{1:r} >\beta_{1:r}$, this implies that $\beta_{r:r} >
    \gamma_{r:r}$ and $\nu > \frac{\gamma_{1:r}-\beta_{1:r}}{r-1} +\mu$. 
    By the rule of signs, $h_1$ has sign variation in $\mathbb{R}_+$ either $-+$ or $-+-+$. When 
    $\frac{M(r,\gamma_{1:r},\mu)}{M(r,\beta_{1:r},\nu)}=1$, 
    then 
    \[
    h''_1(1) = -2  \frac{\gamma_{1:r}-\beta_{1:r}}{r-1} -\mu +\nu 
    =\frac{2}{r-1} \left( \frac{\beta_{1:r}+\beta_{r:r}}{2}-
    \frac{\gamma_{1:r}+\gamma_{r:r}}{2}\right).
    \]
    Consider increasing arithmetic vectors $(\alpha_{1:r},\ldots,
    \alpha_{r:r})$ such that the product of its positive coordinates has a fixed value. 
    One can see that increasing the common difference decreases the mean $\frac{\alpha_{1:r}+\alpha_{r:r}}{2}$. 
    Since $\nu >\mu$, we have $\frac{\beta_{1:r}+\beta_{r:r}}{2} < \frac{\gamma_{1:r}+\gamma_{r:r}}{2}$, and $h''_1(1)<0$ in 
    consequence. It means that both $-+$ and $-+-+$ are possible sign variations of $h_1$ in $(0,1)$. 
    Suppose that the latter holds. It follows that (\ref{5}) is positive on some $(a,b)$ with $0<a<b<1$. 
    Slightly changing the parameters of $m$-\gos\ we violate the assumption $\frac{M(r,\gamma_{1:r},\mu)}{M(r,\beta_{1:r},\nu)}=1$ 
    and notice that $h_1$ is positive on some $(a',b')$ for some $0<a'<b'<1$. 
    This contradicts the hitherto established conclusions. 
    Therefore the sign variation of (\ref{5}) is $-+$ on the unit interval. This gives the desired claim in the case (i). 
    \end{itemize}
\item[(b)]
    The cases (iii) and (iv) 
    with $\mu=0 <\nu$: We now have
    \begin{eqnarray}
    \lefteqn{f_{r,\gamma_{1:r},0}(x) - f_{r,\beta_{1:r},\nu}(x) =
    \frac{\gamma_{1:r}^r}{(r-1)!}(1-x)^{\gamma_1-1} } \nonumber \\
    & & \times \left[\Bigl( -\ln (1-x)\Bigr)^{r-1}
    - \! \frac{M(r,\beta_{1:r},\nu)}{M(r,\gamma_{1:r},0)}
    (1-x)^{\beta_{1:r}-\gamma_{1:r}}\!\left( \frac{1\!-\!(1\!-\!x)^\nu}{\nu}\right)^{r-1}\right].\qquad
    \label{10}
    \end{eqnarray}
Replacing $y=1-x$ in (\ref{10}), we consider the sign variation on $(0,1)$ of
\begin{equation}
\label{11}
h_2(y) \!=\!  -\ln y- \frac{1}{\nu} 
    \left(\frac{M(r,\beta_{1:r},\nu)}{M(r,\gamma_{1:r},0)}\right)^{\frac1{r-1}}
    y^{-\frac{\gamma_{1:r}-\beta_{1:r}}{r-1}}
     + \frac{1}{\nu} 
    \left(\frac{M(r,\beta_{1:r},\nu)}{M(r,\gamma_{1:r},0)}\right)^{\frac1{r-1}}
    y^{-\frac{\gamma_{1:r}-\beta_{1:r}}{r-1}+\nu}.
\end{equation}
%
We have $h_2(0)=-\infty$ and $h_2(1)=0$. According to Lemma~\ref{gen-Descartes}, the sign variation of the derivative
    \begin{eqnarray}
    h'_2(y) & = &  \frac{1}{\nu} \frac{\gamma_{1:r}-\beta_{1:r}}{r-1}
    \left(\frac{M(r,\beta_{1:r},\nu)}{M(r,\gamma_{1:r},0)}\right)^{\frac1{r-1}}
    y^{-\frac{\gamma_{1:r}-\beta_{1:r}}{r-1}-1} -y^{-1} \nonumber \\
    & & + \frac{1}{\nu} \left[- \frac{\gamma_{1:r}-\beta_{1:r}}{r-1}+\nu \right]
    \left(\frac{M(r,\beta_{1:r},\nu)}{M(r,\gamma_{1:r},0)}\right)^{\frac1{r-1}}
    y^{-\frac{\gamma_{1:r}-\beta_{1:r}}{r-1}+\nu-1}\quad
    \label{12}
    \end{eqnarray}
    on $\mathbb{R}_+$ is either $+-$ when $-\frac{\gamma_{1:r}-\beta_{1:r}}{r-1}+\nu\leq0$ (note that, in the case of equality the last term in (\ref{12}) does not exist, so the sign variation is the same as for the inequality) or either of $+$ and $+-+$ when
    $-\frac{\gamma_{1:r}-\beta_{1:r}}{r-1}+\nu>0$, which is equivalent to $\beta_{r:r}>\gamma_{1:r}$.
       \begin{itemize}
        \item
        The case (iii): 
        If $-\frac{\gamma_{1:r}-\beta_{1:r}}{r-1}+\nu\leq0$, it follows that $h'_2(1)<0$, hence the sign variation in $(0,1)$ of (\ref{11}) is $-+$, and, due to replacing $y=1-x$, (\ref{10}) is $+-$ on the unit interval. 
        
        It remains to verify the case $\prod_{i=1}^r \frac{\beta_{1:r}+(i-1)\nu}{\gamma_{1:r}} =1$ which implies $h'_2(1)=0$.
        It is impossible that $h'_2$ is positive on $\mathbb{R}_+$, hence its sign variation on $\mathbb{R}_+$ is $+-+$. Therefore, $h_2$ first has a unique local maximum and then a unique local minimum in $\mathbb{R}_+$. Since $h_2(1)= h'_2(1)=0$, the function has a local extremum at $y=1$. If this is a local maximum, then $h_2$ is increasing and negative on $(0,1)$, which is impossible. If it is a local minimum, then $h_2$ is increasing-decreasing, and first negative and then positive in $(0,1)$. This means that (\ref{10}) has sign variation $+-$ on this interval.

        \item
        The case (iv): 
        Assume now that $0 < -\frac{\gamma_{1:r}-\beta_{1:r}}{r-1}+\nu $, which is equivalent to $\beta_{1:r}> \gamma_{1:r}$ as mentioned above. We may compute
        \begin{equation*}
        h'_2(1) =  \left(\frac{M(r,\beta_{1:r},\nu)}{M(r,\gamma_{1:r},0)}\right)^{\frac1{r-1}}
        -1.
        \label{13}
        \end{equation*}
        If $\frac{M(r,\beta_{1:r},\nu)}{M(r,\gamma_{1:r},0)}>1$,
        i.e., $h'_2(1)>0$ and $h_2(1-) <0$, the function in (\ref{11}) is $-+-$ in $(0,1)$ and so is (\ref{10}) there (negativity of these functions on $(0,1)$ is impossible).     
        \end{itemize}

\item[(c)]
    The case (v): 
    Now we consider the case $\nu=0 <\mu$. We have
    \begin{eqnarray}
    \lefteqn{f_{r,\gamma_{1:r},\mu}(x) \!-f_{r,\beta_{1:r},0}(x)  =
    \frac{M(r,\beta_{1:r},\nu)}{(r-1)!}
    (1-x)^{\beta_{1:r}-1} } \nonumber \\
     & & \times\!
     \left[ \frac{M(r,\gamma_{1:r},\mu)}{M(r,\beta_{1:r},\nu)}
     (1-x)^{\gamma_{1:r}-
    \beta_{1:r}}\!\left( \frac{1\!-\!(1\!-\!x)^\mu}{\mu}\right)^{r-1}
    \!-\! \left( -\ln (1\!-\!x)\right)^{r-1}\! \right]\!.\;\;
    \label{7}
    \end{eqnarray}
    For checking the sign of (\ref{7}) we analyze the sign of
    \begin{equation}
    h_3(y) = \left(\frac{M(r,\gamma_{1:r},\mu)}{M(r,\beta_{1:r},\nu)}\right)^{\frac1{r-1}}
    y^{\frac{\gamma_{1:r}-\beta_{1:r}}{r-1}}
    \frac{1-y^\mu}{\mu} + \ln y .
    \label{8}
    \end{equation}
    Observe that $h_3(0)=-\infty$ and $h_3(1)=0$, and
    \begin{eqnarray}
    h'_3(y) & = & y^{-1} +\frac{1}{\mu}\frac{\gamma_{1:r}-\beta_{1:r}}{r-1}
    \left(\frac{M(r,\gamma_{1:r},\mu)}{M(r,\beta_{1:r},\nu)}\right)^{\frac1{r-1}}
    y^{\frac{\gamma_{1:r}-\beta_{1:r}}{r-1}-1} \nonumber \\
    & & -
    \frac{1}{\mu}\left(\frac{\gamma_{1:r}-\beta_{1:r}}{r-1}+\mu\right)
    \left(\frac{M(r,\gamma_{1:r},\mu)}{M(r,\beta_{1:r},\nu)}\right)^{\frac1{r-1}}
    y^{\frac{\gamma_{1:r}-\beta_{1:r}}{r-1}+\mu -1} .
    \label{9}
    \end{eqnarray}
    Taking into account Lemma~\ref{gen-Descartes}, (\ref{9}) has sign variation $+-$ on $\mathbb{R}_+$, hence (\ref{8}) is first increasing and then decreasing there.
    By the border conditions, it is either increasing and negative on $(0,1)$, or it is increasing negative and positive and ultimately decreasing positive there.
    We conclude that (\ref{7}) is either negative or negative-positive on $(0,1)$.
    The first possibility is excluded, because the integral of (\ref{7}) over $(0,1)$ is equal to $0$. This means that (\ref{7}) is $+-$ on $(0,1)$.
    
    \smallskip
    
    To complete the proof it remains to verify the case $\mu=\nu=0$, where we have
    \begin{equation}
    f_{r,\gamma_{1:r},0}(x) - f_{r,\beta_{1:r},0}(x)  
    = 
    \frac{\beta^r_{1:r}}{(r\!-\!1)!}(1-x)^{\beta_{1:r}-1}(-\ln (1-x))^{r-1}
    \left[ \frac{\gamma^r_{1:r}}{\beta^r_{1:r}}(1\!-\!x)^{\gamma_{1:r}-\beta_{1:r}}-1
    \right].
    \label{14}
    \end{equation}
    The factor in front of the expression in the square brackets is positive, while the one inside the brackets is decreasing from
    $\frac{\gamma^r_{1:r}}{\beta^r_{1:r}}-1 >0$ at $0$ to $-1$ at $1$, which means that (\ref{14}) changes the sign once from $+$ to $-$
    on $(0,1)$.
\end{enumerate}
\end{proof}
We now study the sign variation of the difference between the densities of two $m$-\gos\ with different orders and the same common difference.
\begin{lemma}\label{t11}
Let $f_{r,\gamma_{1:r},\mu}$ and $f_{q,\beta_{1:q},\mu}$ denote the density functions of the $r$-th  and $q$-th, $r \neq q$, uniform $m$-\gos, respectively,  
with the minimal parameters $\gamma_{1:r} > \beta_{1:r}>0$, respectively, and identical common differences $\mu \geq 0$. 
If $r>q$ then the difference $f_{r,\gamma_{1:r}, \mu} - f_{q, \beta_{1:q}, \mu}$ has the sign variation $-+-$ in $(0,1)$. If $r<q$, then this difference has sign variation $+-$ in $(0,1)$.
\end{lemma} 
\begin{proof}
Suppose that $\mu>0$ and $r>q$. Then
\begin{eqnarray}
\label{14a}
\lefteqn{f_{r, \gamma_{1:r}, \mu} (x) - f_{q, \beta_{1:q}, \mu} (x)
=  f_{q, \beta_{1:r}, \mu} (x)} \nonumber \\
& \qquad\times & \left[
\frac{M(r,\gamma_{1:r},\mu)}{M(q,\beta_{1:q},\nu)}
\frac{(q-1)!}{(r-1)!}\frac{1}{\mu^{r-q}}
(1-x)^{\gamma_{1:r}-\beta_{1:r}} \bigl[1- (1-x)^\mu\bigr]^{r-q}-1\right]\!.
\end{eqnarray}
Under change of variables $y= (1-x)^\mu$ the expression in the second line has the form\linebreak 
$c y^{\frac{\gamma_{1:r}-\beta_{1:r}}{\mu}}(1-y)^{r-q}-1$ for a positive $c$. It takes the value $-1$ when $y=0$ or $y=1$, and it is increasing-decreasing in between, because its derivative
\[
c y^{\frac{\gamma_{1:r}-\beta_{1:r}}{\mu}-1}
(1-y)^{r-q-1}\left[ \frac{\gamma_{1:r}-\beta_{1:r}}{\mu} (1-y) -
(r-q)y \right]
\]
is first positive and then negative.
Consequently, (\ref{14a}) is either negative, positive and negative in $(0,1)$ or merely negative there. However, the latter case is impossible.

\smallskip

Consider now the case when $\mu=0$ and $r>q$. We now have
\begin{equation}
\label{14b}
f_{r, \gamma_{1:r}, 0} (x) - f_{q, \beta_{1:r}, 0} (x)
  =  f_{q, \beta_{1:r}, 0} (x)   \left[ \frac{\gamma_{1:r}^r}{\beta_{1:r}^q}
\frac{(q\!-\!1)!}{(r\!-\!1)!} (1-x)^{\gamma_{1:r}-\beta_{1:r}}
(-\ln (1-x))^{r-q}-1 \right].
\end{equation}
The expression in the brackets is equal to $-1$ at $0$ and $1$, and it is increasing-decreasing in $(0,1)$, because its
derivative
\[
\frac{\gamma_{1:r}^r}{\beta_{1:r}^q}
\frac{(q-1)!}{(r-1)!} (1-x)^{\gamma_{1:r}-\beta_{1:r}-1}
\Bigl[-\ln (1-x)]^{r-q-1}[(\gamma_{1:r}-\beta_{1:r})\ln (1-x) +r-q\Bigr]
\]
is first positive and then negative.  In conclusion, the sign variation of (\ref{14b}) is $-+-$.

\smallskip

We proceed now to the case $r<q$. For $\mu>0$ yields
\begin{eqnarray}
\lefteqn{f_{r, \gamma_{1:r}, \mu} (x) - f_{q, \beta_{1:q}, \mu} (x)
=  \frac{M(r,\beta_{1:r},\nu)}
{(q-1)!\mu^{r-1}}
(1-x)^{\beta_{1:r}-1} [1-(1-x))^\mu ]^{r-1}} \nonumber \\
& & \qquad \times \left[ 
 \frac{M(r,\gamma_{1:r},\mu)}{M(q,\beta_{1:q},\nu)}
\frac{(q-1)!}{(r-1)!}
(1-x)^{\gamma_{1:r}-\beta_{1:q}} - \frac{1}{\mu^{q-r}}
[1-(1-x)^\mu ]^{q-r} \right].
\label{14c}
\end{eqnarray}
The first term in the brackets decreases from a positive number when $x=0$ to $0$ when $x=1$, whereas the following subtrahend increases from $0$ when $x=0$ to a positive value when $x=1$. Therefore (\ref{14c}) is $+-$ on $(0,1)$, as claimed.

We finally consider the case $\mu=0$ with $r<q$. We have
\begin{eqnarray}
\lefteqn{f_{r, \gamma_{1:r}, 0} (x) - f_{q, \beta_{1:q},0} (x)
=  \frac{\beta_{1:r}^q}{(q-1)!}
(1-x)^{\beta_{1:q}-1} [-\ln(1-x)]^{r-1}} \nonumber \\
& & \qquad\qquad\quad \times \left[ \frac{\gamma_{1:r}^r}{\beta_{1:q}^q}
\frac{(q-1)!}{(r-1)!}
(1-x)^{\gamma_{1:r}-\beta_{1:q}} -
\bigl[-\ln (1-x) \bigr]^{q-r} \right].
\label{14d}
\end{eqnarray}
The first term in the second line decreases from a positive number at $0$ to $0$ at $1$, and the latter one increases from $0$ to $+\infty$
on $(0,1)$. Therefore the sign variation of (\ref{14d}) is $+-$. This completes the proof of Lemma \ref{t11}.
\end{proof}

\begin{remark}
\label{rem0}
Define the principal branch of the Lambert function as $\mathcal W$, that is, $y=\mathcal W(x)$ denotes the solution of $y e^y=x$, where $x\geq 0.$ It is worth noting that the unique zero of the difference of density functions in the case when $\mu=\nu=0$ can be expressed in terms of values of $\mathcal W$.
Indeed, it is easily seen that, with $u=-\ln(1-x)$, the root of (\ref{14d}) satisfies
$$
u=\tfrac{q-r}{\gamma_{1:r}-\beta_{1:q}}\mathcal{W}\left(\tfrac{\gamma_{1:r}-\beta_{1:q}}{q-r}
\left(\tfrac{\gamma_{1:r}^r} {\beta_{1:q}^q}
\tfrac{(q-1)!}{(r-1)!}\right)^{-1/(r-q)}\right).
$$
We will use this representation in the simplified framework of $k$-th records in Section~\ref{sec:records}.
\end{remark}

The previous lemmas, together with Theorem~\ref{t0}, immediately imply some stochastic dominance results.
\begin{corollary}\label{c1}
Let $X_{r, \tilde{\gamma}_r}$ and $X_{r,\tilde{\beta}_r}$ denote the $r$-th $m$-\gos\ 
with parameter vectors $\tilde{\gamma}_r= (\gamma_1,\ldots,\gamma_r)$ and $\tilde{\beta}_r= (\beta_1,\ldots ,\beta_r)$ such that
$\gamma_{1:r} \leq \cdots \leq \gamma_{r:r}$ and $\beta_{1:r} \leq \cdots \leq \beta_{r:r}$ form arithmetic sequences with common differences $\mu$ and $\nu$, respectively, and the same baseline distribution function $F$. Assume further that $\gamma_{1:r} >\beta_{1:r}$ together with one of the conditions (i), (iii) or (v) 
of Lemma~\ref{t1}, or $\gamma_{1:r} = \beta_{1:r}$  with $\mu \geq \nu$ hold.  Then $X_{r, \tilde{\gamma}_r} \preceq_{st} X_{r,\tilde{\beta}_r}$.
\end{corollary}

\begin{proof}
Suppose first that $\gamma_{1:r} > \beta_{1:r}$. Under the conditions (i), (iii) or (v) 
of Lemma~\ref{t1} the density difference (\ref{4}) has sign variation $+-$ on the common support interval $(0,1)$. Hence the difference of the respective distribution functions 
is first increasing and then decreasing there. Combining this fact with vanishing of the difference at the interval 
ends $0$ and $1$ we conclude that  $F_{r,\gamma_{1:r},\mu}>F_{r,\beta_{1:r},\nu}$ on $(0,1)$. 
In other words, $U_{r, \tilde{\gamma}_r} \preceq_{st} U_{r,\tilde{\beta}_r}$. Since any nondecreasing transformation of random variables preserves 
the usual stochastic order we also have $X_{r, \tilde{\gamma}_r} \stackrel{d}{=} F^{-1}( U_{r, \tilde{\gamma}_r}) \preceq_{st}F^{-1}( U_{r,\tilde{\beta}_r}) \stackrel{d}{=}  X_{r,\tilde{\beta}_r}$. 

\smallskip

In the case $\gamma_{1:r} = \beta_{1:r}$ with $\mu \geq \nu$ we have $\gamma_{i:r} \geq \beta_{i:r}$, $i=1,\ldots,r$, which implies $U_{r, \tilde{\gamma}_r} \preceq_{st} U_{r,\tilde{\beta}_r}$ and $X_{r, \tilde{\gamma}_r} \preceq_{st} X_{r,\tilde{\beta}_r}$ in consequence. To check this fact it suffices to note that increasing any $\gamma_i$ in (\ref{1}) results in decreasing $ U_{r, \tilde{\gamma}_r}$.
\end{proof}

\begin{corollary}\label{c11}
Suppose that  $X_{r, \tilde{\gamma}_r}$ and $X_{q,\tilde{\beta}_q}$ are the \gos\ with parameter vectors $\tilde{\gamma}_r= (\gamma_{1:r},\ldots,\gamma_r)$ and $\tilde{\beta}_q= (\beta_1,\ldots ,\beta_q)$ such that  $\gamma_{1:r} \leq \cdots \leq \gamma_{r:r}$ and  $\beta_{1:r} \leq \cdots \leq \beta_{q:q}$ form arithmetic sequences with identical common differences $\mu$, and  common baseline distribution functions $F$. 
If $\gamma_{1:r} \geq \beta_{1:q}$ and $r \leq q$, then $X_{r, \tilde{\gamma}_r} \preceq_{st} X_{q,\tilde{\beta}_q}$.
\end{corollary}

\begin{proof} If $\gamma_{1:r} > \beta_{1:q}$ and $r < q$, we get the statement using the arguments of the first part of the previous corollary. If $r=q$ and $\gamma_{1:r} \geq \beta_{1:q}$ we have $\gamma_{i:r} \geq \beta_{i:q}$, $i=1,\ldots,r$, and relation 
$X_{r, \tilde{\gamma}_r} \preceq_{st} X_{r,\tilde{\beta}_r}$  follows for $\beta_{1:q}, \ldots,\beta_{r:q}$ being the ordered coordinates of 
$\tilde{\beta}_r$. If $r<q$ and $\gamma_{1:r} = \beta_{1:q}$ we simply observe $X_{r, \tilde{\gamma}_r} \preceq_{st} X_{r,\tilde{\beta}_r} \preceq_{st} X_{q,\tilde{\beta}_q}$.
\end{proof}


\begin{remark}\label{rem1}
The following extensions of Corollaries~\ref{c1} and \ref{c11} are immediate:
    \begin{enumerate}
    \item
    by changing the roles of the parameter vectors $\tilde{\gamma}_r$ and $\tilde{\beta}_q$ (and $\tilde{\beta}_r$ in particular) we obtain analogous results in the cases $\gamma_{1:r} \leq \beta_{1:q}$;
    
    \item
    replacing $X_{r, \tilde{\gamma}_r}$ and $X_{q,\tilde{\beta}_q}$ by $X_{p,\tilde{\gamma}_p}$, and $X_{s,\tilde{\beta}_s}$, respectively, where $p<r$ with $\tilde{\gamma}_p$ consists of a choice of any $p$ elements of $\tilde{\gamma}_r$, and $s>q$, 
    $\tilde{\beta}_s=(\beta_1, \ldots, \beta_q, \beta_{q+1}, \ldots, \beta_s)$ with arbitrary $\beta_{s+1},\ldots,\beta_s>0$ and baseline distribution function $G$ of $X_{s,\tilde{\beta}_s}$ satisfying $G \succeq_{st} F$ we get $X_{p,\tilde{\gamma}_p} 
    \preceq_{st} X_{s, \tilde{\beta}_s}$.
    \end{enumerate}
\end{remark}


\subsection{Increasing convex and increasing concave orders}

Clearly the usual stochastic orders proved under the conditions of the Corollaries \ref{c1} and \ref{c11} and Remark \ref{rem1}
imply the increasing convex, increasing concave and star-shaped orderings among the above $m$-generalized order statistics. For establishing the increasing convex and concave orders assuming other shape relations on the baseline distribution, we use the following lemma.

\begin{lemma}[\cite{shaked2007}, Theorem 4.A.22(b)] \label{l0}
If two random variables $X$ and $Y$ have finite expectations satisfying $\mathbb{E}X \leq\mathbb{E}Y$ and distribution functions $F$ and $G$, respectively,
such that the difference $G-F$ ($F-G$) are first positive and then negative, then $X \preceq_{icx} Y$ ($X \preceq_{icv} Y$, respectively).
\end{lemma}

We consider first families of distributions that have monotone failure rate, possibly in their generalized version, that is, distributions that can be compared with some generalized Pareto distribution through the convex transform order.
{\samepage
\begin{proposition}\label{c2}
Let $X_{r,\gamma_{1:r},\mu}$ and $X_{r,\beta_{1:r},\nu}$ denote the $r$-th $m$-\gos\ with minimal parameters $\gamma_{1:r}$, $\beta_{1:r}$,
respectively, common differences $\mu$ and $\nu$, respectively, and a common baseline distribution function $F$. Assume that either
\begin{equation*}
\label{18}
\gamma_{1:r} > \beta_{1:r},  \qquad \mu,\nu >0, \qquad \prod_{i=1}^r
\frac{M(r,\gamma_{1:r},\mu)}{M(r,\beta_{1:r},\nu)}<1,
\end{equation*}
or
\begin{equation*}
\label{19}
\beta_{r:r}= \beta_{1:r}+(r-1)\nu > \gamma_{1:r} > \beta_{1:r},
\qquad \mu= 0 < \nu,
\qquad 
\frac{M(r,\gamma_{1:r},0)}{M(r,\beta_{1:r},\nu)}<1.
\end{equation*}
Under either of conditions
\begin{enumerate}[(i)]
\item
$F\preceq_c W_0$ ($F$ is IFR) and
\begin{equation}
\label{20}
\sum_{i=1}^r \frac{1}{\gamma_{1:r}+ (i-1)\mu} \leq
\sum_{i=1}^r \frac{1}{\beta_{1:r}+ (i-1)\nu},
\end{equation}

\item
$F\preceq_c W_\alpha$ for some $\alpha >0$ ($F$ is $\alpha$-IGFR) and
\begin{equation}
\label{21}
\frac{M(r,\gamma_{1:r},\mu)}{M(r,\alpha+\gamma_{1:r},\mu)} \geq \frac{M(r,\beta_{1:r},\mu)}{M(r,\alpha+\beta_{1:r},\mu)},
\end{equation}

\item
$F\preceq_c W_\alpha$ for some $-\beta_{1:r} < \alpha <0$
($F$ is $\alpha$-IGFR) and
\begin{equation}
\label{22}
\frac{M(r,\gamma_{1:r},\mu)}{M(r,\alpha+\gamma_{1:r},\mu)} \leq \frac{M(r,\beta_{1:r},\mu)}{M(r,\alpha+\beta_{1:r},\mu)},
\end{equation}
\end{enumerate}
we have $X_{r,\gamma_{1:r},\mu} \preceq_{icv}X_{r,\beta_{1:r},\nu}$.
\end{proposition}}

\begin{proof}
Suppose that the conditions of the latter claim of Lemma~\ref{t1} hold. Then the sign variation of (\ref{4}) is $-+-$.
The respective distribution function difference first decreases from
$0$ at $0$ to a negative minimum, then increases and eventually decreases to $0$ at $1$. 
This necessarily means that the maximal value is positive, and the difference of the distributions functions is first negative and then positive on $(0,1)$.

\smallskip

We apply Theorem~\ref{t0} in combination with Lemma~\ref{l0}. By Lemma~\ref{t1} the differences $F_{r,\gamma_{1:r},\mu}- F_{r, \beta_{1:r},\nu}$ of distribution functions of $r$-th uniform $m$-\gos\ $U_{r,\gamma_{1:r}, \mu}$ and $U_{r,\beta_{1:r},\nu}$  are first negative and then positive.
The same sign variation holds for the difference 
of distribution functions of the $m$-\gos\ $Y_{r,\gamma_{1:r}, \mu,\alpha} = W_\alpha^{-1}(U_{r,\gamma_{1:r}, \mu})$
and $Y_{r,\beta_{1:r},\nu,\alpha}  =W_\alpha^{-1}(U_{r,\beta_{1:r},\nu})$ with the baseline distribution functions $W_\alpha$, $\alpha \in \mathbb{R}$.
For the expected values, we have
\begin{eqnarray}
\label{23a}
\mathbb{E}Y_{r,\tilde{\gamma}_r,\alpha} & = &
\left\{ 
\begin{array}{ll}
\sum_{i=1}^r \frac{1}{\gamma_{i}}, & \alpha=0, \\[1.75ex]
\frac{1}{\alpha} \left[ 1- 
\prod_{i=1}^r \frac{\gamma_i}{\alpha+\gamma_i} 
\right], 
&  \alpha\neq 0
\end{array}
\right.
\end{eqnarray}
in the general set-up, and
\begin{equation}
\label{23b}
\mathbb{E}Y_{r,\gamma_{1:r}, \mu,\alpha} = \left\{
\begin{array}{ll}
\sum_{i=1}^r \frac{1}{\gamma_{1:r}+(i-1)\mu}, & \alpha=0, \\[1.75ex]
\frac{1}{\alpha}\left[ 1- 
\frac{M(r,\gamma_{1:r},\mu)}{M(r,\alpha+\gamma_{1:r},\mu)}
\right],
& \alpha\neq 0,
\end{array}
\right.
\end{equation}
for the $m$-\gos\ submodel with  $\alpha > -\gamma_{1:r}$ in both the cases.
Conditions (\ref{20}), (\ref{21}) and (\ref{22}) are equivalent
to $\mathbb{E}Y_{r,\gamma_{1:r}, \mu,\alpha} \leq \mathbb{E}Y_{r,\beta_{1:r}, \nu,\alpha}$ in the cases $\alpha= 0$, $\alpha>0$, and $\alpha<0$,  respectively.
By Lemma~\ref{l0} it follows that $Y_{r,\gamma_{1:r}, \mu,\alpha} \preceq_{icv} Y_{r,\beta_{1:r}, \nu,\alpha}$.
Due to Theorem~\ref{t0}, under the assumption $F \preceq_c W_\alpha$, we also have $X_{r,\gamma_{1:r}, \mu} \preceq_{icv}X_{r,\beta_{1:r}, \nu}$ for $m$-\gos\ with baseline distribution function~$F$.
%
\end{proof} 


In a similar way, we prove the following.
\begin{proposition}\label{c3}
Suppose that either
\begin{equation*}
\label{24}
\gamma_{1:r} < \beta_{1:r},  \qquad \mu,\nu >0, \qquad
\frac{M(r,\gamma_{1:r},\mu)}{M(r,\beta_{1:r},\nu)}>1,
\end{equation*}
or
\begin{equation*}
\label{25}
\gamma_{r:r} > \beta_{1:r} > \gamma_{1:r}, \qquad \nu= 0 < \mu,
\qquad 
\frac{M(r,\beta_{1:r},0)}{M(r,\gamma_{1:r},\mu)}<1.
\end{equation*}
hold. Under either of conditions
    \begin{enumerate}[(i)]
    \item
    $F\succeq_c W_0$ ($F$ is DFR) with (\ref{20}),

    \item
    $F\succeq_c W_\alpha$ for some $\alpha >0$ ($F$ is $\alpha$-DGFR)
    with (\ref{21}),

    \item
    $F\succeq_c W_\alpha$ for some $-\gamma_{1:r} < \alpha <0$
    ($F$ is $\alpha$-DGFR) with (\ref{22}),
    \end{enumerate}
we get $X_{r,\gamma_{1:r},\mu} \preceq_{icx}X_{r,\beta_{1:r},\nu}$ with the notation described in Proposition {\rm \ref{c2}}.

\end{proposition}


The statements of Lemma~\ref{t11} allow us to formulate the following conclusions. The proofs are similar to the proofs of
Propositions \ref{c2}--\ref{c3}.

\begin{proposition}\label{c2a}
Assume the $X_{r,\gamma_{1:r}, \mu}$ and  $X_{q,\beta_{1:q},\mu}$, with $r > q$ and $\gamma_{1:r} > \beta_{1:q}$,
denote the $r$-th and $q$-th $m$-\gos\ with minimal parameters  $\gamma_{1:r}$ and $\beta_{1:r}$, respectively, identical common differences $\mu$, and identical baseline distribution function $F$.
Under either of conditions
    \begin{enumerate}[(i)]
    \item
    $F\preceq_c W_0$ ($F$ is IFR) and
    \begin{equation}
    \label{20a}
    \sum_{i=1}^r \frac{1}{\gamma_{1:r}+ (i-1)\mu} \leq
    \sum_{i=1}^q \frac{1}{\beta_{1:q}+ (i-1)\mu},
    \end{equation}

    \item
    $F\preceq_c W_\alpha$ for some $\alpha >0$ ($F$ is $\alpha$-IGFR) and
    \begin{equation}
    \label{21a}
    \frac{M(r,\gamma_{1:r},\mu)}{M(r,\alpha+\gamma_{1:r},\mu)} \geq \frac{M(q,\beta_{1:q},\mu)}{M(q,\alpha+\beta_{1:q},\mu)},
    \end{equation}

    \item
    $F\preceq_c W_\alpha$ for some $-\beta_{1:q} < \alpha <0$
    ($F$ is $\alpha$-IGFR) and
    \begin{equation}
    \label{22a}
    \frac{M(r,\gamma_{1:r},\mu)}{M(r,\alpha+\gamma_{1:r},\mu)} \leq \frac{M(q,\beta_{1:q},\mu)}{M(q,\alpha+\beta_{1:q},\mu)},
    \end{equation}
    \end{enumerate}
we obtain $X_{r,\gamma_{1:r},\mu} \preceq_{icv}X_{q,\beta_{1:q},\mu}$.
\end{proposition}

Note that, as remarked after the introduction of the generalized Pareto family (see (\ref{1b})), cases (iii) in Propositions~\ref{c2} and \ref{c3} include the DOR and IOR families, respectively, by taking $\alpha=-1$.

\begin{proposition}\label{c3a}
Assume the notation of Proposition~\ref{c1} with  $r<q$ and $\gamma_{1:r} < \beta_{1:q}$. Under either of conditions
    \begin{enumerate}[(i)]
    \item
    $F\succeq_c W_0$ ($F$ is DFR) and (\ref{20a}),

    \item
    $F\succeq_c W_\alpha$ for some $\alpha >0$ ($F$ is $\alpha$-DGFR) and (\ref{21a}),

    \item
    $F\succeq_c W_\alpha$ for some $-\gamma_{1:r} < \alpha <0$  ($F$ is $\alpha$-DGFR) and (\ref{22a}),
    \end{enumerate}
we get $X_{r,\gamma_{1:r},\mu} \preceq_{icx}X_{q,\beta_{1:q},\mu}$. 
\end{proposition}

\begin{remark}\label{rem2}
The statements of Propositions \ref{c2}--\ref{c3a} can easily be extended. 
For instance, if we take the assumptions of Proposition \ref{c2}, and $X_{p, \tilde{\gamma}_p}$ and $X_{s,\tilde{\beta}_s}$ as in Remark~\ref{rem1} with the modification that the respective baseline distribution functions $F$ and $G$ satisfy 
$F \preceq_{icx} G$ then $X_{p, \tilde{\gamma}_p} \preceq_{icx} X_{s,\tilde{\beta}_s}$.
\end{remark}

We now prove $\preceq_{icx}$ and $\preceq_{icv}$ results assuming the monotonicity of the generalized reverse failure rate. For this purpose, we will be comparing the baseline distribution with the negative generalized Pareto distributions $\widetilde{W}_\alpha$.
\begin{proposition}
	\label{prop:neg}
	Let $X_{r,\gamma_{1:r},\mu}$ and $X_{q,\beta_{1:q},\nu}$ be {$m$-\gos} based on the same baseline $F$.
	\begin{enumerate}[(a)]
    \item
    Assume that parameter conditions of Proposition~\ref{c2} (when $q=r$) or of Proposition~\ref{c2a} (when $\mu=\nu$) hold. Suppose that $F\preceq_c \widetilde W_\alpha$ for some $\alpha\in\mathbb R$.
	\begin{enumerate}[(i)]
		\item If $\alpha=0$ and
		\[
		\sum_{\ell=1}^\infty \frac{1}{\ell}\,
\frac{M(r,\gamma_{1:r},\mu)}{M(r,\ell+\gamma_{1:r},\mu)}
		\ \ge\
		\sum_{\ell=1}^\infty \frac{1}{\ell}\,\frac{M(q,\beta_{1:q},\nu)}{M(q,\ell+\beta_{1:q},\nu)},
		\]
		then $X_{r,\gamma_{1:r},\mu}\preceq_{icv}X_{q,\beta_{1:q},\nu}$.
		
		\item If $\alpha>0$ and
		\[
		\sum_{\ell=0}^\infty
        \binom{\alpha}{\ell}(-1)^\ell\,\frac{M(r,\gamma_{1:r},\mu)}{M(r,\ell+\gamma_{1:r},\mu)} \ \le\
		\sum_{\ell=0}^\infty \binom{\alpha}{\ell}(-1)^\ell\, 
        \frac{M(q,\beta_{1:q},\nu)}{M(q,\ell+\beta_{1:q},\nu)},
		\]
		then $X_{r,\gamma_{1:r},\mu}\preceq_{icv}X_{q,\beta_{1:q},\nu}$.
		
		\item If $\alpha<0$ with $\alpha>-1$ and
		\[
		\sum_{\ell=0}^\infty \binom{\alpha}{\ell}(-1)^\ell\,
        \frac{M(r,\gamma_{1:r},\mu)}{M(r,\ell+\gamma_{1:r},\mu)}
		\ \ge\
		\sum_{\ell=0}^\infty \binom{\alpha}{\ell}(-1)^\ell\,
        \frac{M(q,\beta_{1:q},\nu)}{M(q,\ell+\beta_{1:q},\nu)},
		\]
		then $X_{r,\gamma_{1:r},\mu}\preceq_{icv}X_{q,\beta_{1:q},\nu}$.
	   \end{enumerate}

    \item
    Assume now that the parameter conditions of Proposition~\ref{c3} (when $q=r$) or of Proposition~\ref{c3a} (when $\mu=\nu$) hold. If $F\succeq_c \widetilde W_\alpha$, then conditions (i), (ii), (iii) above imply $X_{r,\gamma_{1:r},\mu}\preceq_{icv}X_{q,\beta_{1:q},\nu}$.
\end{enumerate}
\end{proposition}
\begin{proof}	
    We prove part \textit{(a)}, as the second part is analogous using the second conclusion in Corollary~\ref{l0}. 
    Recall that $\gamma_{i:r}=\gamma_{1:r}+(i-1)\mu$. Define, for $\alpha\neq0,$
	$$
	\widetilde Y_{r,\gamma_{1:r},\mu,\alpha}
	:=\widetilde
    W_\alpha^{-1}\!\bigl(U_{r,\gamma_{1:r},\mu}\bigr)=\frac{U_{r,\gamma_{1:r},\mu}^{\alpha}-1}{\alpha}.
	$$
	Using the representation of the uniform $m$-\gos,
	$U_{r,\gamma_{1:r},\mu}\ \stackrel d=\ 
    1-\prod_{i=1}^r U_i^{1/\gamma_{i:r}}$,
	where $U_i$ are i.i.d.\ uniformly distributed in $[0,1]$.
    Then, for every $\alpha$ for which the expectation exists, bearing in mind that a product of uniform random variables is always in $[0,1],$ the binomial theorem gives
	$$
    \mathbb E U_{r,\gamma_{1:r},\mu}^{\alpha}
    =\sum_{\ell=0}^{\infty}\binom{\alpha}{\ell}(-1)^{\ell}\,\mathbb E\bigg[\prod_{i=1}^r U_i^{\ell/\gamma_{i:r}}\bigg].
    $$
    (For case (iii), $\alpha>-1$ ensures that $\mathbb{E} U_{r,\gamma_{1:r},\mu}<\infty$, as required to apply Lemma~\ref{l0}.) By independence,
	$$
	\mathbb E\bigg[\prod_{i=1}^r U_i^{\ell/\gamma_{i:r}}\bigg]
	=\prod_{i=1}^r \mathbb E\bigg[U_i^{\ell/\gamma_{i:r}}\bigg]
    =\frac{M(r,\gamma_{1:r},\mu)}{M(r,\ell+\gamma_{1:r},\mu)},
	$$
	which gives
	\begin{equation}
    \label{eq:expectation}\mathbb E U_{r,\gamma_{1:r},\mu}^{\alpha}
	=\sum_{\ell=0}^{\infty}\binom{\alpha}{\ell}(-1)^{\ell}\,
	\frac{M(r,\gamma_{1:r},\mu)}{M(r,\ell+\gamma_{1:r},\mu)}.
    \end{equation}
	Finally, we obtain
	$$
	\mathbb E \widetilde Y_{r,\gamma_{1:r},\mu,\alpha}
	=\frac{1}{\alpha}\left(
	\sum_{\ell=0}^{\infty}\binom{\alpha}{\ell}(-1)^{\ell}\,
	\frac{M(r,\gamma_{1:r},\mu)}{M(r,\ell+\gamma_{1:r},\mu)}
    -\;1\right),
	\qquad \alpha\neq 0.
    $$
    In the special case $\alpha\in\mathbb N$, the sum ends at $\ell=\alpha$.
	
    \smallskip
    
	For $\alpha=0$, $\widetilde Y_{r,\gamma_{1:r},\mu,0}
	=\ln U_{r,\gamma_{1:r},\mu}.$
	In this case we use the power series expansion
	$$
    \ln(1-z)=-\sum_{\ell=1}^\infty \frac{z^\ell}{\ell},\qquad 0<z<1,
    $$
	which yields
\begin{equation}
\label{eq:expectation0}	
\mathbb E\Bigl[\ln U_{r,\gamma_{1:r},\mu}\Bigr]
	=\mathbb E\left[\ln\left(1-\prod_{i=1}^r U_i^{1/\gamma_{i:r}}\right)\right]
	=-\sum_{\ell=1}^\infty \frac{1}{\ell}\,
    \mathbb{E}\bigg[\prod_{i=1}^r U_i^{\ell/\gamma_{i:r}}\bigg].
\end{equation}
	The conclusion follows similarly to the case $\alpha\neq0$.
\end{proof}
\begin{remark}
One can easily check that the expectations of 
$m$-\gos\ for negative generalized Pareto baseline distributions with 
negative parameter $\alpha$ are finite for $\alpha >-1$ when $\mu>0$
and $\alpha >-r$ when $\mu=0$. 
We show that the series in (\ref{eq:expectation}) and (\ref{eq:expectation0}) converge under these conditions and so they
represent the respective expectations.
Note that
$$
\frac{M(r,\gamma_{1:r},\mu)}{M(r,\ell+\gamma_{1:r},\mu)}
=\prod_{i=1}^r \frac{\gamma_{1:r}+(i-1)\mu}{\ell+\gamma_{1:r}+(i-1)\mu}\leq	\left(\prod_{i=1}^r(\gamma_{1:r}+(i-1)\mu)\right)\ell^{-r}.
$$
Moreover, the extended binomial coefficients may be represented as
$$
\binom{\alpha}{\ell}=(-1)^\ell\,\frac{\Gamma(\ell-\alpha)}{\Gamma(-\alpha)\,\Gamma(\ell+1)},
$$
hence $\bigl|\binom{\alpha}{\ell}\bigr|\sim\ell^{-\alpha-1}$, implying that the general term in the expansion (\ref{eq:expectation}) behaves like $\ell^{-(r+\alpha+1)}$, thus it is convergent when $\alpha>-r$ (and similarly $\alpha>-q$, when considering $\mathbb E U_{q,\gamma_{1:q},\nu}^{\alpha}$), which is satisfied as $r,q\geq 1$.

In a similar way we show that the summands in (\ref{eq:expectation0}) are of order $\ell^{-(r+1)}$.
\end{remark}

Conditions for the $\preceq_{icx}$ and $\preceq_{icv}$ ordering of the $m$-\gos\ for the class of distributions with monotone odds rate may also be characterized, as this monotonocity is expressed through the convex transform ordering with respect to $L(x)=(1+e^{-x})^{-1}$, for $x\in\mathbb{R}$, the logistic distribution function. Indeed, a distribution $F$ has increasing (decreasing) log-odds rate, ILOR for short (DLOR, respectively) if $L\succeq_c (\preceq_c)F$. 

\begin{proposition}
\label{prop:log-odds}
Let $X_{r,\gamma_{1:r}, \mu}$ and  $X_{r,\beta_{1:r},\nu}$, denote the $r$-th $m$-\gos\ with minimal parameters  $\gamma_{1:r}$ and $\beta_{1:r}$, respectively, identical common differences $\mu$ and $\nu$, respectively, and common baseline distribution function $F$. Recall that $\gamma_{i:r}=\gamma_{1:r}+(i-1)\mu$ and $\beta_{i:r}=\beta_{1:r}+(i-1)\nu$.
    \begin{enumerate}[(i)]
    \item
    If $F \preceq_c L$ ($F$ is ILOR),
    $\gamma_{1:r}\geq\beta_{1:r}$, $\frac{M(r,\gamma_{1:r},\mu)}{M(r,\beta_{1:r},\nu)} < 1$ and
    \begin{equation}
    \label{eq:icx-ior}
    \sum_{i=1}^r \left(\frac{1}{\gamma_{i:r}}-\frac{1}{\beta_{i:r}}\right)
    \leq
    \sum_{\ell=1}^\infty\frac1\ell\left(
    \frac{M(r,\gamma_{1:r},\mu)}{M(r,\ell+\gamma_{1:r},\mu)}
    -
    \frac{M(r,\beta_{1:r},\mu)}{M(r,\ell+\beta_{1:r},\mu)}\right).
    \end{equation}
    Then $X_{r,\gamma_{1:r}, \mu} \leq_{icv} X_{r,\beta_{1:r},\nu}$.
    \item
    If $F \succeq_c L$ ($F$ is DLOR), 
    $\gamma_{1:r}\geq\beta_{1:r}$, 
    $\frac{M(r,\gamma_{1:r},\mu)}{M(r,\beta_{1:r},\nu)} < 1$ and
    \begin{equation}
    \label{eq:icv-dor}
    \sum_{i=1}^r \left(\frac{1}{\gamma_{i:r}}-\frac{1}{\beta_{i:r}}\right)
    \geq
    \sum_{\ell=1}^\infty\frac1\ell\left(
    \frac{M(r,\gamma_{1:r},\mu)}{M(r,\ell+\gamma_{1:r},\mu)}
    -
    \frac{M(r,\beta_{1:r},\mu)}{M(r,\ell+\beta_{1:r},\mu)}\right).
    \end{equation}
    Then $X_{r,\gamma_{1:r}, \mu} \leq_{icx} X_{r,\beta_{1:r},\nu}$.
    \end{enumerate}
\end{proposition}
\begin{proof}
This is again a consequence of combining Theorem~\ref{t0} with Lemma~\ref{t1}. In the case (i), $F\preceq_c L$, hence the proof reduces to comparing $\mathbb{E}(L^{-1}(U_{r,\gamma_{1:r},\mu}))$ and $\mathbb{E}(L^{-1}(U_{r,\beta_{1:r},\nu}))$, where $U_{r,\gamma_{1:r},\mu}$ and $U_{r,\beta_{1:r},\nu}$  are the $r$-th uniform $m$-\gos. According to Lemma~\ref{t1}(ii), the assumption on the parameters imply that $f_{r,\gamma_{1:r},\mu}-f_{r,\beta_{1:r},\nu}$ has sign variation $-+-$, hence $F_{r,\gamma_{1:r},\mu}-F_{r,\beta_{1:r},\nu}$ has sign variation $-+$. We may compute, with $U_i$ i.i.d.\ uniform $(0,1)$,
\begin{eqnarray*}
\mathbb{E}(L^{-1}(U_{r,\gamma_{1:r},\mu})) & = & 
   \mathbb{E}\left[\ln\left(1-\prod_{i=1}^r U_i^{1/\gamma_{i:r}}\right)\right]
 -\mathbb{E}\left[\ln\left(\prod_{i=1}^r U_i^{1/\gamma_{i:r}}\right)\right]  \\[1.25ex]
  & = &  -\sum_{\ell=1}^\infty\frac1\ell \frac{M(r,\gamma_{1:r},\mu)}{M(r,\ell+\gamma_{1:r},\mu)}
+ \sum_{i=1}^r\frac1{\gamma_{i:r}}.
\end{eqnarray*}
The inequality $\mathbb{E}(L^{-1}(U_{r,\gamma_{1:r},\mu}))<\mathbb{E}(L^{-1}(U_{r,\beta_{1:r},\nu}))$ 
is then equivalent to (\ref{eq:icx-ior}).
\end{proof}
A similar argument allows for a result about different sized $m$-\gos. We state the result without proof, as it goes along the same arguments as above, referring to Lemma~\ref{t11} instead of Lemma~\ref{t1}.
\begin{proposition}
\label{prop:log-odds1}
Let $X_{r,\gamma_{1:r}, \mu}$ and  $X_{q,\beta_{1:q},\mu}$, denote the $r$-th and $q$-th $m$-\gos\ with $r>q$ and minimal parameters  $\gamma_{1:r}$ and $\beta_{1:q}$, respectively, identical common differences $\mu$, and common baseline distribution function $F$.
    \begin{enumerate}[(i)]
    \item
    If $F \preceq_c L$ ($F$ is ILOR),
    $\gamma_{1:r}\geq\beta_{1:q}$ and
    \begin{equation*}
    \label{eq:icx-ior-rq}
    \sum_{i=1}^r \frac{1}{\gamma_{i:r}} -\sum_{i=1}^q\frac{1}{\beta_{i:q}}
    \leq
    \sum_{\ell=1}^\infty\frac1\ell\left(
    \frac{M(r,\gamma_{1:r},\mu)}{M(r,\ell+\gamma_{1:r},\mu)}
    -
    \frac{M(r,\beta_{1:q},\mu)}{M(r,\ell+\beta_{1:q},\mu)}\right).
    \end{equation*}
    Then $X_{r,\gamma_{1:r}, \mu} \leq_{icv} X_{q,\beta_{1:q},\mu}$.
    \item
    If $F \succeq_c L$ ($F$ is DLOR),
    $\gamma_{1:r}\geq\beta_{1:r}$ and
    \begin{equation*}
    \label{eq:icv-dor-rq}
    \sum_{i=1}^r \frac{1}{\gamma_{i:r}}
    -\sum_{i=1}^q\frac{1}{\beta_{i:q}}
    \geq
    \sum_{\ell=1}^\infty\frac1\ell\left(
    \frac{M(r,\gamma_{1:r},\mu)}{M(r,\ell+\gamma_{1:r},\mu)}
    -
    \frac{M(r,\beta_{1:q},\mu)}{M(r,\ell+\beta_{1:q},\mu)}\right).
    \end{equation*}
    Then $X_{r,\gamma_{1:r}, \mu} \geq_{icx} X_{q,\beta_{1:q},\mu}$.
    \end{enumerate}
\end{proposition}

\subsection{Star-shaped orders}
For verifying the star-shaped order among $m$-\gos\ with baseline distributions drawn from restricted families, we apply the following lemma.
\begin{lemma}[\cite{shaked2007}, Theorem 4.A.54]
\label{l2}
Two life distribution functions
satisfy $F \preceq_{ss} G$  iff $\int_y^\infty x \,F(dx) \leq
\int_y^\infty x \,G(dx)$ for every $y \geq 0$.
\end{lemma}
In order to apply Lemma~\ref{l2} we determine explicit formulae for $\int_y^\infty x \,F(dx)$ for the distribution functions of $m$-\gos\ with baseline generalized Pareto distributions.

We first consider the case of exponential baseline distribution and $r$-th $m$-\gos\ with nonzero common difference $\mu$.
Integrating by parts, we calculate 
\begin{eqnarray}
\lefteqn{Z_{r,\gamma_{1;r}, \mu,0}(y) =
\int_y^\infty x \,F_{r,\gamma_{1;r}, \mu} (W_0 (dx)) }
\nonumber \\
& = &
 - x \bigl[1-F_{r,\gamma_{1:r}, \mu}(W_0(x))\bigr] \Bigr|_y^\infty
+ \int_y^\infty 1-F_{r,\gamma_{1:r}, \mu}(W_0(x))\,dx \nonumber \\
& = & e^{-\gamma_{1:r}y}\sum_{j=0}^{r-1}
\frac{M(j,\gamma_{1:r},\mu)}{j!\mu^j}
\left[y (1-e^{-\mu y})^j +
\sum_{k=0}^j {j \choose k} (-1)^k
\frac{e^{-k\mu y}}{\gamma_{1:r}+k\mu}\right]. \nonumber
\label{25b}
\end{eqnarray}
for all $y \geq 0$.
Similarly, for $\mu=0$, we obtain
\begin{eqnarray}
\lefteqn{Z_{r,\gamma_{1;r}, 0,0}(y) =
\int_y^\infty x \,F_{r,\gamma_{1;r}, 0} (W_0 (dx)) }
\nonumber \\ & = &
y \bigr[1-F_{r,\gamma_{1:r}, 0}(W_0(y))\bigl]
+ \int_y^\infty 1-F_{r,\gamma_{1:r}, 0}(W_0(x))\,dx \nonumber \\
& = & e^{-\gamma_{1:r}y}
\sum_{j=0}^{r-1}
\left[\frac{\gamma^j_{1:r}}{j!} y^{j+1}
+  \sum_{k=0}^j \frac{\gamma_{1:r}^{k-1}}{k!} y^k
\right] . \nonumber
\label{25c}
\end{eqnarray}
%
%
For $\alpha \neq 0 \neq \mu$  with $d_\alpha= \infty$ for $\alpha <0$
and $\frac{1}{\alpha}$ for $\alpha >0$ under the change of variables
$z= (1-\alpha x)^{\frac{\mu}{\alpha}}$,
we have
\begin{eqnarray}
\lefteqn{Z_{r,\gamma_{1;r}, \mu,\alpha}(y)=
\int_y^\infty x \,F_{r,\gamma_{1;r}, \mu} (W_\alpha (dx)) }
\nonumber \\ 
& = &
y \bigl[1-F_{r,\gamma_{1:r}, \mu}(W_\alpha(y))\bigr]
+ \int_y^{d_\alpha} 1-F_{r,\gamma_{1:r}, \mu} (W_\alpha(x))\,dx \nonumber \\
& = & \sum_{j=0}^{r-1}
\frac{M(j,\gamma_{1:r},\mu)}{j!\mu^j} \nonumber \\
& & 
\times\left[ y(1-\alpha y)^{\frac{\gamma_{1:r}}{\alpha}}
\left(1-(1-\alpha y)^{\frac{\mu}{\alpha}}\right)^j +
\sum_{k=0}^j {j \choose k} (-1)^k
\frac{(1-\alpha y)^{\frac{\alpha+ \gamma_{1:r} +k\mu}{\alpha}}}{
\alpha+ \gamma_{1:r} +k\mu}
\right]. \nonumber
\label{25d}
\end{eqnarray}
Under the same notation for $\alpha \neq 0 = \mu$ we obtain
\begin{eqnarray}
\lefteqn{Z_{r,\gamma_{1;r}, 0,\alpha}(y)=
\int_y^\infty x \,F_{r,\gamma_{1;r},0} (W_\alpha (dx)) }
\nonumber \\ 
& = & \sum_{j=0}^{r-1}
\frac{\gamma_{1:r}^j}{j!}\left[
y(1\!-\!\alpha y)^{\frac{\gamma_{1:r}}{\alpha}} \left( - \frac{\ln (1\!-\!\alpha y)}{\alpha}
\right)^j + \int_y^{d_\alpha}
(1\!-\!\alpha x)^{\frac{\gamma_{1:r}}{\alpha}} \left( - \frac{
\ln (1\!-\!\alpha x)}{\alpha}
\right)^j \,dx\! \right] \nonumber \\[1.25ex]
& = & \sum_{j=0}^{r-1}
\frac{\gamma_{1:r}^j}{j!}\left[
y(1-\alpha y)^{\frac{\gamma_{1:r}}{\alpha}} \left( - \frac{\ln (1-\alpha y)}{\alpha}
\right)^j\right. 
+ \left. (1-\alpha y)^{\frac{\gamma_{1:r}}{\alpha}+1}
\sum_{k=0}^j \frac{j!(-\ln (1-\alpha y))^k}{k!\alpha^k(\gamma_{1:r}+\alpha)^{j+1-k}}
\right]. \nonumber
\label{25e}
\end{eqnarray}

Now we are in a position to establish the star-shaped order relations for $m$-generalized order statistics with baseline distributions determined by star relations with generalized Pareto distributions.
The cases of $r$-th $m$-\gos\ satisfying conditions (i), (iii) or (v) 
of Lemma~\ref{t1} do not need to be taken into account, because owing to Corollary~\ref{c1}, $X_{r,\gamma_{1:r}, \mu} \preceq_{st}X_{r, \beta_{1:r},\nu}$ and so $X_{r,\gamma_{1:r}, \mu} \preceq_{ss} X_{r, \beta_{1:r},\nu}$. 
Due to Corollary~\ref{c2}, for $\gamma_{1:r}> \beta_{1:q}$ and $r>q$, we have that $X_{r,\gamma_{1:r}, \mu} \preceq_{ss} X_{q, \beta_{1:q},\mu}$.
Moreover, with the notation of Remark~\ref{rem1} with the weaker assumption $F \preceq_{ss} G$ we get
$X_{p,\tilde{\gamma}_p} \preceq_{ss} X_{s,\tilde{\beta}_s}$.
\begin{proposition}\label{p1}
Consider $r$-th $m$-\gos\ $X_{r,\gamma_{1:r},\mu}$ and $X_{r,
\beta_{1:r},\nu}$ with identical parent distribution function $F$.
\begin{enumerate}[(a)]
\item
Let $\gamma_{1:r} > \beta_{1:r}$, $\frac{M(r,\gamma_{1:r},\mu)}{M(r,\beta_{1:r},\nu)}<1$ and $\mu,\nu >0$.
    \begin{enumerate}[(i)]
    \item
    If $F \succeq_* W_0$ ($F$ is  DFRA), and for the first zero $x_*$ of $f_{r, \gamma_{1:r}, \mu} - f_{r,\beta_{1:r}, \nu}$
    in $(0,1)$ we have
    \begin{equation*}
    \label{25f}
    Z_{r, \gamma_{1:r}, \mu, 0}(-\ln (1-x_*)) -
    Z_{r, \beta_{1:r}, \nu, 0}(-\ln (1-x_*)) \leq 0,
    \end{equation*}
    then $X_{r,\gamma_{1:r}, \mu} \preceq_{ss}X_{r, \beta_{1:r},\nu}$.

    \item
    If $-\beta_{1:r} <\alpha\neq 0$,  $F \succeq_* W_\alpha$ ($F$ is $\alpha$-DGFRA), and for the first zero  $x_*$ of $f_{r, \gamma_{1:r}, \mu} - f_{r,\beta_{1:r}, \nu}$ we have
    \begin{equation*}
    \label{25g}
    Z_{r, \gamma_{1:r}, \mu, \alpha}\left(\frac{1-(1-x_*)^\alpha}{\alpha}
    \right) -
    Z_{r, \beta_{1:r}, \nu, \alpha}\left(\frac{1-(1-x_*)^\alpha}{\alpha}
    \right)  \leq 0,
    \end{equation*}
    then $X_{r,\gamma_{1:r}, \mu} \preceq_{ss} X_{r, \beta_{1:r},\nu}$.
    \end{enumerate}

\item
Let $\gamma_{1:r} > \beta_{1:r}$,  $\beta_{r:r} > \gamma_{1:r}$, 
$\frac{M(r,\gamma_{1:r},0)}{M(r,\beta_{1:r},\nu)}<1$ and $\mu=0 <\nu$.
    \begin{enumerate}[(i)]
    \item
    If $F \succeq_* W_0$ ($F$ is DFRA), and for the first zero $x_*$ of $f_{r, \gamma_{1:r}, 0} - f_{r,\beta_{1:r}, \nu}$ in $(0,1)$ we have
    \begin{equation*}
    \label{25h}
    Z_{r, \gamma_{1:r}, 0, 0}(-\ln (1-x_*)) -
    Z_{r, \beta_{1:r}, \nu, 0}(-\ln (1-x_*)) \leq 0,
    \end{equation*}
    then $X_{r,\gamma_{1:r}, 0} \preceq_{ss}X_{r, \beta_{1:r},\nu}$.

    \item
    If $-\beta_{1:r} <\alpha\neq 0$, and $F \succeq_* W_\alpha$ ($F$ is $\alpha$-DGRFA), and for the first zero $x_*$ of $f_{r, \gamma_{1:r}, 0} - f_{r,\beta_{1:r}, \nu}$ we have
    \begin{equation*}
    \label{25i}
    Z_{r, \gamma_{1:r}, 0, \alpha}\left(\frac{1-(1-x_*)^\alpha}{\alpha}
    \right) -
    Z_{r, \beta_{1:r}, \nu, \alpha}\left(\frac{1-(1-x_*)^\alpha}{\alpha}
    \right)  \leq 0,
    \end{equation*}
    then $X_{r,\gamma_{1:r}, 0} \preceq_{ss} X_{r, \beta_{1:r},\nu}$.
    \end{enumerate}
\end{enumerate}
\end{proposition}

\begin{proof}
We formally prove only Proposition~\ref{p1}(ia). The other cases are verified in such the same way. 
We only mention that the assumption $-\beta_{1:r}<
\alpha <0$ in the cases (iia) and (iib) assert finiteness of expectations of $m$-\gos\ with the baseline distribution functions 
$W_\alpha$ which is necessary for applying Lemma \ref{l2} for these 
distributions.
The difference $Z_{r, \gamma_{1:r}, \mu, 0}(y) -
Z_{r, \beta_{1:r}, \nu, 0}(y)$ tends to $0$ as $x$ tends to $+\infty$. The derivative of the difference equals to
\begin{equation*}
\label{25k}
Z'_{r, \gamma_{1:r}, \mu, 0}(y) -
Z'_{r, \beta_{1:r}, \nu, 0}(y) = y e^{-y} \bigl[
f_{r, \beta_{1:r},\nu}(1-e^{-y}) - f_{r, \gamma_{1:r}, \mu}(1-e^{-y})\bigr],
\end{equation*}
and its sign is identical with the sign of the expression
in the brackets.
By Lemma~\ref{t1}, its sign variation is $+-+$, which means that $Z_{r, \gamma_{1:r}, \mu, 0}(y) - Z_{r, \beta_{1:r}, \nu, 0}(y)$
first increases to a local maximum at $W^{-1}_0(x_*)= -\ln (1-x_*)$, then decreases to a negative local minimum and eventually increases to $0$.
It is always nonpositive if it is nonpositive at $-\ln (1-x_*)$. Due to Lemma \ref{l2}, we get $Y_{r,\gamma_{1:r}, \mu,0}\preceq_{ss} Y_{r,\beta_{1:r}, \nu, 0}$.
Since $F \succeq_* W_0$, by Theorem \ref{t0}, we also have $X_{r,\gamma_{1:r}, \mu}\preceq_{ss} X_{r,\beta_{1:r}, \nu}$
for the respective $m$-\gos\ with the baseline distribution function $F$.
\end{proof} 


Similar arguments lead to the following proposition.
\begin{proposition}\label{p2}
Let $X_{r, \gamma_{1:r}, \mu}$ and $X_{q, \beta_{1:q}, \mu}$ denote the $r$-th and $q$-th $m$-\gos\ with the same common difference $\mu$
and baseline distribution function $F$. Suppose that $\gamma_{1:r} > \beta_{1:q}$ and $r > q$.
\begin{enumerate}
\item
Assume $\mu=0$.
    \begin{enumerate}[(i)]
    \item
    If 
    $F \succeq_* W_0$ ($F$ is DFRA),
    and for the first zero  $x_*$ of $f_{r, \gamma_{1:r}, 0} - f_{q,\beta_{1:q}, 0}$ on $(0,1)$ we have
    \begin{equation*}
    \label{25l}
    Z_{r, \gamma_{1:r}, 0, 0}(-\ln (1-x_*)) -
    Z_{q, \beta_{1:q}, 0, 0}(-\ln (1-x_*)) \leq 0,
    \end{equation*}
    then $X_{r,\gamma_{1:r}, 0} \preceq_{ss} X_{q, \beta_{1:q},0}$.

    \item
    If 
    $-\beta_{1:q} < \alpha \neq 0$,
    $F \succeq_* W_\alpha$ ($F$ is $\alpha$-DGFRA), and for the first zero $x_*$ of $f_{r, \gamma_{1:r}, 0} - f_{q,\beta_{1:q}, 0}$ on $(0,1)$ we have
    \begin{equation*}
    \label{25m}
    Z_{r, \gamma_{1:r}, 0, \alpha}\left(\frac{1-(1-x_*)^\alpha}{\alpha}
    \right)  -
    Z_{q, \beta_{1:q}, 0, \alpha}\left(\frac{1-(1-x_*)^\alpha}{\alpha}
    \right)  \leq 0,
    \end{equation*}
    then $X_{r,\gamma_{1:r}, 0} \preceq_{ss} X_{q, \beta_{1:q},0}$.
    \end{enumerate}

\item
Take now $\mu>0$.
    \begin{enumerate}[(i)]
    \item
    If 
    $F \succeq_* W_0$ ($F$ is DFRA),
    and for the first zero
    $x_*$ of $f_{r, \gamma_{1:r}, \mu} - f_{q,\beta_{1:q}, \mu}$ on $(0,1)$, we have
    \begin{equation*}
    \label{25n}
    Z_{r, \gamma_{1:r}, \mu, 0}(-\ln (1-x_*)) -
    Z_{q, \beta_{1:q}, \mu, 0}(-\ln (1-x_*)) \leq 0,
    \end{equation*}
    then $X_{r,\gamma_{1:r}, 0} \preceq_{ss} X_{q, \beta_{1:q},0}$.
    
    \item
    If
    $-\beta_{1:q} < \alpha \neq 0$,
    $F \succeq_* W_\alpha$ ($F$ is $\alpha$-DGFRA), and for the first zero $x_*$ of $f_{r, \gamma_{1:r}, \mu} - f_{q,\beta_{1:q}, \mu}$ on $(0,1)$ we have
    \begin{equation*}
    \label{25p}
    Z_{r, \gamma_{1:r}, \mu, \alpha}\left(\frac{1-(1-x_*)^\alpha}{\alpha}
    \right)  -
    Z_{q, \beta_{1:q}, \mu, \alpha}\left(\frac{1-(1-x_*)^\alpha}{\alpha}
    \right)  \leq 0,
    \end{equation*}
    then $X_{r,\gamma_{1:r}, 0} \preceq_{ss} X_{q, \beta_{1:q},0}$.
    \end{enumerate}
\end{enumerate}
\end{proposition}

\begin{remark}\label{rem3}
Natural generalizations of Propositions~\ref{p1} and \ref{p2} follow for $X_{p,\tilde{\gamma}_p}$ and $X_{s, \tilde{\beta}_s}$ defined as in 
Remark~\ref{rem1} as long as the respective baseline distribution functions satisfy $F \preceq_{ss} G$, thus concluding, in this setting, that $X_{p,\tilde{\gamma}_p} \preceq_{ss}X_{s, \tilde{\beta}_s}$.
\end{remark}

Note that, under the assumptions of either of Propositions~\ref{p1} or \ref{p2}, $f_{r,\gamma_{1:r},\mu}-f_{q,\beta_{1:q},\nu}$ has sign variation $+-+$ on $(0,1)$.
Let
\[
\widetilde Z_{r,\gamma_{1:r}, \mu,\alpha}(y)
:=
\int_y^0 x \,F_{r,\gamma_{1:r}, \mu} \bigl(\widetilde W_\alpha (dx)\bigr),
\qquad y<0,
\]
and note that, by change of variables 
$$
\widetilde Z_{r,\gamma_{1:r},\mu,\alpha}(y)=
-\int_0^{-y} t \,F_{r,\gamma_{1:r}, \mu} \bigl(W_\alpha(dt)\bigr)=
Z_{r,\gamma_{1:r},\mu,\alpha}(-y)\;-\;Z_{r,\gamma_{1:r},\mu,\alpha}(0),
$$
where $Z_{r,\gamma_{1:r}, \mu, \alpha}(0)=E Y_{r,\gamma_{1:r}, \mu, \alpha}$, is described in (\ref{23b}).
{\samepage
\begin{proposition}
\label{negSS}
Let $X_{r,\gamma_{1:r},\mu}$ and $X_{q,\beta_{1:q},\nu}$ be $m$-\gos\ based on the same baseline distribution function $F$.
Assume that the corresponding parameter conditions of Proposition~\ref{p1} (when $q=r$) or Proposition~\ref{p2} (when $\mu=\nu$) hold, and let $x_*\in(0,1)$ be the first zero of $f_{r,\gamma_{1:r},\mu}-f_{q,\beta_{1:q},\nu}$.
Assume one of the following conditions:
    \begin{enumerate}[(i)]
    \item 
    $\alpha=0$, $F\succeq_* \widetilde W_0$ and
    \begin{equation*}
    Z_{r,\gamma_{1:r},\mu,0}(-\log x_*)-Z_{q,\beta_{1:q},\nu,0}(-\log x_*)
    \ \le\
    \sum_{i=1}^r \frac{1}{\gamma_{1:r}+(i-1)\mu}
    -
    \sum_{i=1}^q \frac{1}{\beta_{1:q}+(i-1)\nu},
    \end{equation*}	
    
    \item 
    $\alpha\neq 0$, $F\succeq_* \widetilde W_\alpha$ and
    \begin{equation*}
    Z_{r,\gamma_{1:r},\mu,\alpha}\left(\frac{1-x_*^\alpha}{\alpha}\right)-
    Z_{q,\beta_{1:q},\nu,\alpha}\left(\frac{1-x_*^\alpha}{\alpha}\right)
    \ \leq\
    \frac{1}{\alpha}\left[
    \frac{M(q,\beta_{1:q},\nu)}{M(q,\alpha+\beta_{1:q},\nu)}
    -
    \frac{M(r,\gamma_{1:r},\mu)}{M(r,\alpha+\gamma_{1:r},\mu)}
    \right].
    \end{equation*}
    \end{enumerate}
Then
$X_{r,\gamma_{1:r},\mu}\ \preceq_{ss}\ X_{q,\beta_{1:q},\nu}$.
\end{proposition}}

\section{Order statistics and $k$-th record values}
\label{sec:records}

Below, we specify the results of the previous section  to the particular models of order and record statistics.
We first consider order statistics $X_{i:n}$ and $X_{j:m}$  from two independent i.i.d.\ samples of sizes $n$ and $m$, respectively,
with a common parent distribution function.
In this case, we particularly take on $r=i$, $q=j$, $\gamma_{1:r}=n+1-i$, $\beta_{1:q}= m+1-j$ and $\mu=\nu=1$.
From Corollary~\ref{c11}, we conclude that $X_{i:n} \preceq_{st}X_{j:m}$ if $i<j$ and $n-i>m-j$.
Below, we state the translations of Propositions \ref{c2a}, \ref{c3a}, and \ref{p2}(b) into the order statistcs framework.
\begin{corollary}\label{coos1}
Let $i>j$ and $n-i >m-j$. Under either of the following conditions
    \begin{enumerate}[(a)]
    \item
    $F$ is IFR and
    \begin{equation}
    \label{26}
    \sum_{k=1}^i \frac{1}{n-i+k} \leq \sum_{k=1}^j \frac{1}{m-j+k},
    \end{equation}

    \item
    $F$ is $\alpha$-IGFR for some $\alpha >0$ and
    \begin{equation}
    \label{27}
    \prod_{k=1}^i \frac{n-i+k}{\alpha+n-i+k} \geq
    \prod_{k=1}^j \frac{m-j+k}{\alpha+m-j+k},
    \end{equation}

    \item
    $F$ is $\alpha$-IGFR for some $-m-1+j < \alpha<0$ and
    \begin{equation}
    \label{28}
    \prod_{k=1}^i \frac{n-i+k}{\alpha+n-i+k} \leq
    \prod_{k=1}^j \frac{m-j+k}{\alpha+m-j+k},
    \end{equation}
    \end{enumerate}
the relation $X_{i:n} \preceq_{icv} X_{j:m}$ holds.
\end{corollary}

{\samepage
\begin{corollary}\label{coos2}
Let $i<j$ and $n-i >m-j$, and either of the following conditions holds:
    \begin{enumerate}[(a)]
    \item $F$ is DFR and (\ref{26}),
    \item $F$ is $\alpha$-DGFR for some $\alpha >0$ and (\ref{27}),
    \item $F$ is $\alpha$-DGFR for some $-m-1+j <\alpha <0$ and (\ref{28}).
    \end{enumerate}
Then $X_{i:n} \preceq_{icx} X_{j:m}$.
\end{corollary}
}

{\samepage
\begin{corollary}\label{coos3}
Assume $i>j$ and $n-i > m-j$. Let $x_*$ denote the first zero of the difference $f_{i,n+1-i,1} - f_{j,m+1-j,1}$ in $(0,1)$ (cf.\ (\ref{2a}) and (\ref{2b})),
and either of the following assumptions is valid
\begin{enumerate}[(a)]
\item
$F$ is DFRA, and
\begin{equation}
\label{29}
Z_{i,n+1-i,1,0} \left( - \ln(1-x_*)\right)
-
Z_{j,m+1-j,1,0} \left( - \ln(1-x_*)\right)
\leq 0,
\end{equation}

\item
$F$ is $\alpha$-DGFRA for some $\alpha >0$, and
\begin{equation}
\label{30}
Z_{i,n+1-i,1,\alpha} \left(  \frac{1-(1-x_*)^\alpha}{\alpha}\right)
-
Z_{j,m+1-j,1,\alpha} \left(  \frac{1-(1-x_*)^\alpha}{\alpha}\right)
\leq 0,
\end{equation}

\item
$F$ is $\alpha$-DGFRA for some $-m-1+j < \alpha < 0$ and (\ref{30}). 
\end{enumerate}
Then $X_{i:n} \preceq_{ss} X_{j:m}$ holds.
\end{corollary}}
Note that the function $f_{i,n+1-i,1}$ is the density function of $i$-th order statistics form the standard uniform i.i.d.\ sample of size $n$ and then (\ref{2a}) with (\ref{2b}) can be simplified to
\begin{equation*}
\label{31}
f_{i,n+1-i,1}(x) = f_{i:n}(x) = n {n-1 \choose i-1} x^{i-1} (1-x)^{n-i}, \qquad 0<x<1,
\end{equation*}
and $f_{j,m+1-j,1}$ has a similar representation.
The above results for $\alpha=-1$, $0$, and $1$ were earlier determined by \cite{ALO-JAP}.

For comparison of record values $R_n^{(k)}$ and $R_m^{(j)}$ from i.i.d.\ sequences with the same marginal distribution, we set $r=n$, $q=m$, $\gamma_{1:r}=k$, $\beta_{1:q}=j$, and $\mu=\nu =0$.
Corollary \ref{c11} asserts that $R_n^{(k)} \preceq_{st}R_m^{(j)}$ if $n<m$ and $k>j$, which is obvious in view of the definitions of  $k$-th record values.
Moreover, we have the following.
\begin{corollary}
\label{cor1}
Assume $n>m$ and $k>j$. If either of the conditions
    \begin{enumerate}[(a)]
    \item
    $F$ is IFR and
    \begin{equation}
    \frac{n}{k} \leq \frac{m}{j},
    \label{32}
    \end{equation}

    \item
    $F$ is $\alpha$-IGFR for some $\alpha >0$ and
    \begin{equation}
    \label{33}
    \left( \frac{k}{\alpha+k} \right)^n \geq \left( \frac{j}{\alpha+j} \right)^m
    \end{equation}

    \item
    $F$ is $\alpha$-IGFR for some $-j < \alpha <0$ and
    \begin{equation}
    \label{34}
    \left( \frac{k}{\alpha+k} \right)^n \leq \left( \frac{j}{\alpha+j} \right)^m
    \end{equation}
    \end{enumerate}
holds, then $R_n^{(k)} \preceq_{icv} R_m^{(j)}$.
\end{corollary}

{\samepage
\begin{corollary}\label{cor2}
Let $n<m$ and $k<j$, and suppose that either of the following assumptions is satisfied:
    \begin{enumerate}[(a)]
    \item $F$ is DFR and (\ref{32}),
    \item $F$ is $\alpha$-DGFR for some positive $\alpha$ and (\ref{33}),
    \item $F$ is $\alpha$-DGFR for some $-k < \alpha <0$, and (\ref{34}).
    \end{enumerate}
Then $R_n^{(k)} \preceq_{icx} R_m^{(j)}$.
\end{corollary}
}

\begin{corollary}\label{cor3}
Suppose that $n>m$ and $k>j$, and $x_*$ is the first zero of the difference $f_{n,k,0} -f_{m,j,0}$. 
Moreover, assume one of the following conditions:
    \begin{enumerate}[(a)]
    \item
    $F$ is DFRA, 
    and
    \begin{equation}
    \label{35}
    Z_{n,k,0,0}(-\ln(1-x_*)) - Z_{m,j,0,0}(-\ln(1-x_*)) \leq 0,
    \end{equation}

    \item
    $F$ is $\alpha$-DGFRA for some $\alpha >0$, 
    and
    \begin{equation}
    \label{36}
    Z_{n,k,0,\alpha}\left(  \frac{1-(1-x_*)^\alpha}{\alpha}\right) -
    Z_{m,j,0,\alpha}\left(  \frac{1-(1-x_*)^\alpha}{\alpha}\right)
    \leq 0,
    \end{equation}

    \item $F$ is $\alpha$-DGFRA for some $-j <\alpha < 0$,  and 
    (\ref{36}). 
    \end{enumerate}
It follows that   $R_n^{(k)} \preceq_{ss} R_m^{(j)}$.
\end{corollary}

\begin{remark}
With respect to $x_*$ in Corollary~\ref{cor3}, substituting in the expressions obtained in Remark~\ref{rem0}, we immediately find the more explicit representation based on Lambert $\mathcal{W}$ function:
$$
x_*=1-e^{-u_*}, \quad\text{where }u_*= -\frac{n-m}{k-j}\,\mathcal W\!\left(-\frac{k-j}{n-m}\bigg(\frac{j^m(n-1)!}{k^n(m-1)!}\bigg)^{\tfrac 1{n-m}}\right).
$$
\end{remark}

\begin{example}
Let $F$ be some DFR distribution. As the DFR class is included in the DFRA class, we may apply Corollaries~\ref{cor2} and \ref{cor3} to obtain ordering relations between different record values. As Corollary~\ref{cor3} does not provide a closed formula for identifying the parameters, we shall look at a particular case. Take ($n,k)=(10,5)$. Applying Corollary~\ref{cor2}(i), one gets $R_{10}^{(5)} \preceq_{icx} R_m^{(j)}$ whenever $m>10$, $j>5$ and $m>2j$. This identifies the shaded area in Figure~\ref{fig1}. To this region we may add the points that verify condition (i) in Corollary~\ref{cor3}, represented by the points above the line $j=\frac{m}2$, where $R_{10}^{(5)} \preceq_{ss} R_m^{(j)}$ will hold.
\begin{figure}[h]
\centering
\begin{tikzpicture}[scale=.25]
\draw[->] (-1,0) -- (20,0) coordinate[label={below:{\scriptsize $m$}}];
\draw[->] (0,-1) -- (0,15) coordinate[label={left:{\scriptsize $j$}}];
\draw (10,-.2) coordinate[label={below:{\scriptsize $10$}}] -- (10,.2);
\draw (-.2,5) coordinate[label={left:{\scriptsize $5$}}] -- (.2,5);
\draw[fill=gray!20,gray!20] (10,5) -- (21,5) -- (21,21/2) -- (10,5);
\draw (0,0) -- (21,21/2) coordinate[label={right:{\scriptsize $j=\frac{m}2$}}];
\draw[dashed] (0,0) -- (15,15) coordinate[label={right:{\scriptsize $j=m$}}];
\node at (17.5,6.6) {{\scriptsize $R_{10}^{(5)} \preceq_{icx} R_m^{(j)}$}};
\node at (1,1) {{\scriptsize $\bullet$}};
\node at (2,1) {{\scriptsize $\bullet$}};
\node at (2,2) {{\scriptsize $\bullet$}};
\node at (3,1) {{\scriptsize $\bullet$}};
\node at (3,2) {{\scriptsize $\bullet$}};
\node at (3,3) {{\scriptsize $\bullet$}};
\node at (4,2) {{\scriptsize $\bullet$}};
\node at (4,3) {{\scriptsize $\bullet$}};
\node at (4,4) {{\scriptsize $\bullet$}};
\node at (5,2) {{\scriptsize $\bullet$}};
\node at (5,3) {{\scriptsize $\bullet$}};
\node at (5,4) {{\scriptsize $\bullet$}};
\node at (6,3) {{\scriptsize $\bullet$}};
\node at (6,4) {{\scriptsize $\bullet$}};
\node at (7,4) {{\scriptsize $\bullet$}};
\node at (8,4) {{\scriptsize $\bullet$}};
\end{tikzpicture}
\caption{Stochastic ordering between $R_{10}^{(5)}$ and $R_m^{(j)}$: $(m,j)$ in the gray region verify $R_{10}^{(5)} \preceq_{icx} R_m^{(j)}$, $(m,j)$ for the remaining dots verify $R_{10}^{(5)} \preceq_{ss} R_m^{(j)}$.}
\label{fig1}
\end{figure}
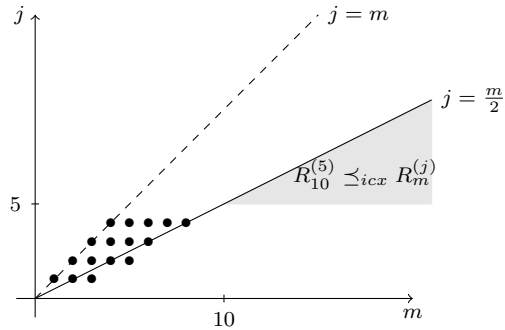
\end{example}

\section{Bounds for the probability of exceeding a \gos}
In this section, we extend the results of Section 7 in \cite{ALO-JAP}. We are interested in the probability that the underlying random variable exceeds an expected \gos, that is, $P(X\geq \mathbb E X_{r,\tilde{\gamma}_r})$. In contrast to Section~\ref{sec:m-gos}, the results here concern the full model of \gos\  admitting arbitrary positive parameters.

The framework covers several practically relevant models. For example, in the case of regular order statistics, when the order statistic $X_{k:n}$ may represent a lifetime of a $k$-out-of-$n$ system with i.i.d.\ components, one can establish bounds for the probability that the lifetime of one component exceeds the expected lifetime of the system \cite{ALO-JAP}. 
The same perspective applies to progressively censored type II order statistics; under a censoring scheme determined by the vector $\mathbf R$, $X_{k:n:N}^{\mathbf R}$ represents the time of the $k$-th failure in a test that starts with $N$ items and stops after $n$ failures ($n\leq N$). In this setting, we are interested in the probability that a single item lasts longer than the expected lifetime of the $k$-th failure, according to the censoring plan $\mathbf R$. In the record value case, the problem reduces to studying the probability of exceeding the $r$-th record $R_r^{(k)}$ of order $k$, which is particularly natural when record values are used to model sequential extremes. Similar exceedance probabilities arise when $\mathbb E X_{r,\tilde{\gamma}_r}$ represents some threshold induced by some sampling plan.

Assuming that $G^{-1}\circ F$ is convex and applying Jensen's inequality we immediately find the bound
$$
G^{-1}\circ F(\mathbb E X_{r,\tilde{\gamma}_r})
= G^{-1} \circ F \,\Bigl(\mathbb{E} F^{-1} (U_{r, \tilde{\gamma}_r})\Bigr)
\leq \mathbb E G^{-1}\left(1-\prod_{i=1}^{r}U_i^{1/\gamma_{i:r}}\right),
$$
so that
$$
\mathbb{P}(X\leq \mathbb E X_{r,\tilde{\gamma}_r})
= G^{-1} \circ G \circ F \,\Bigl(\mathbb{E} F^{-1} (U_{r, \tilde{\gamma}_r})\Bigr)
\leq G\left[\mathbb E G^{-1}\left(1-\prod_{i=1}^{r}U_i^{1/\gamma_{i:r}}\right)\right].
$$
The bound above depends just on $G$, $r$, and $\tilde{\gamma}_r$,  and the baseline distribution $F$ that needs to belong to some convex-ordered family of distributions.
A similar argument applies to the case when $G^{-1}\circ F$ is concave.  The results are summarized in the following proposition.
These types of bounds have been used to obtain convexity or concavity tests in \cite{LandoMohammed2025}.
\begin{proposition}
Given some distribution function $G$, define
$$
\pi_{r,\tilde{\gamma}_r}^G=G\left[\mathbb E G^{-1}\left(1-\prod_{i=1}^{r}U_i^{1/\gamma_{i:r}}\right)\right].
$$
    \begin{enumerate}[(a)]
    \item If $G^{-1}\circ F$ is convex, $P(X\geq \mathbb E X_{r,\tilde{\gamma}_r})\geq 1-\pi_{r,\tilde{\gamma}_r}^G$.
    \item If $G^{-1}\circ F$ is concave, $P(X\geq \mathbb E X_{r,\tilde{\gamma}_r})\leq 1-\pi_{r,\tilde{\gamma}_r}^G$.
    \end{enumerate}
\end{proposition}

Common choices of $G$ yield explicit expressions for the bounds. For example, let $G=W_\alpha$. Then we have 
\[
\mathbb E\,W_\alpha^{-1}\left(1-\prod_{i=1}^{r}U_i^{1/\gamma_{i:r}}\right) =
\left\{ 
\begin{array}{ll}
\sum_{i=1}^r \frac{1}{\gamma_{i:r}}, & \alpha=0, \\[1.5ex]
\frac{1}{\alpha} \left[ 1-  
   \prod_{i=1}^r \frac{\gamma_{i:r}}{\alpha+\gamma_{i:r}}
\right], 
&  -\gamma_{i:r} < \alpha\neq 0,
\end{array}
\right.
\]
(see (\ref{23a})) and so 
\begin{equation}\label{explicit-pi-gos}
\pi^{W_\alpha}_{r,\tilde{\gamma}_r}	 = W_\alpha \left(
\mathbb E\,W_\alpha^{-1}\left(1-\prod_{i=1}^{r}U_i^{1/\gamma_{i:r}}\right) \right) =
\left\{ 
\begin{array}{ll}
1- \exp \left( -\sum_{i=1}^r \frac{1}{\gamma_{i:r}}\right), & \alpha=0,\\[1.75ex]
1-\left(
\prod_{i=1}^r \frac{\gamma_{i:r}}{\alpha+\gamma_{i:r}} 
\right)^{1/\alpha}, & -\gamma_{i:r} < \alpha \neq 0. 
\end{array}
\right.
\end{equation}
For the $k$-th record values (where $\gamma_i=k$), \eqref{explicit-pi-gos} becomes 
\[ 
\pi^{W_\alpha}_{r,k}
=
\left\{
\begin{array}{ll}
1-\exp \left(-\frac{r}{k}\right), & \alpha=0, \\[1.5ex]
1-\left(\frac{k}{k+\alpha}\right)^{r/\alpha}, & -k <\alpha \neq 0.
\end{array}
\right.
\]
The above expressions simplify further if we consider the most common special cases, for which $\alpha=-1,0$ and $1$.

The application of these bounds is straightforward. For instance, assuming that $F$ is IFR and DD, we have an upper and lower bound for $P(X\geq \mathbb E R_r)$, that is,
$$
\exp\left(-\frac{r}{k}\right)\leq \mathbb{P}(X\geq \mathbb E R_r^{(k)}) \leq \left(\frac{k}{k+1}\right)^r.
$$
When the conditions are weakened, wider bounds are derived. For example, assuming that $F$ is IOR and DRFR, we now obtain
$$
\left(\frac{k-1}{k}\right)^r\leq \mathbb{P}(X\geq \mathbb E R_r^{(k)}) \leq \exp\left[ -\sum_{i=1}^\infty \frac 1n\bigg(\frac k{k+n}\bigg)^r\right].
$$
The left bound above is trivial for $k=1$.


\end{document}